\documentclass[11pt,reqno]{amsart}
\usepackage{amsfonts}
\usepackage{amssymb}
\usepackage{a4wide}
\usepackage{pgf,pgfarrows,pgfautomata,pgfheaps,pgfnodes,pgfshade}
\usepackage{url}
\usepackage{lineno}

\newcommand*\patchAmsMathEnvironmentForLineno[1]{%
  \expandafter\let\csname old#1\expandafter\endcsname\csname #1\endcsname
  \expandafter\let\csname oldend#1\expandafter\endcsname\csname end#1\endcsname
  \renewenvironment{#1}%
     {\linenomath\csname old#1\endcsname}%
     {\csname oldend#1\endcsname\endlinenomath}}%
\newcommand*\patchBothAmsMathEnvironmentsForLineno[1]{%
  \patchAmsMathEnvironmentForLineno{#1}%
  \patchAmsMathEnvironmentForLineno{#1*}}%
\AtBeginDocument{%
\patchBothAmsMathEnvironmentsForLineno{equation}%
\patchBothAmsMathEnvironmentsForLineno{align}%
\patchBothAmsMathEnvironmentsForLineno{flalign}%
\patchBothAmsMathEnvironmentsForLineno{alignat}%
\patchBothAmsMathEnvironmentsForLineno{gather}%
\patchBothAmsMathEnvironmentsForLineno{multline}%
}

\usepackage{hyperref}
\numberwithin{equation}{section}

\newtheorem{defi}{Definition}[section]
\newtheorem{thm}[defi]{Theorem}

\newtheorem{rem}[defi]{Remark}
\newtheorem{cor}[defi]{Corollary}

\newtheorem{appli}[defi]{Application}

{%
\setcounter{enumi}{0}

\begin{enumerate}}%
{\end{enumerate} }

\newtheorem{asum}{Assumption}

\DeclareMathOperator{\essup}{ess sup}

\newcommand{\diff}{\,\mathrm{d}}
\newcommand{\diffns}{\mathrm{d}}

\begin{document}

\title[Maximum Principle for Markov Switching FBSDEJ ]{Maximum Principles of Markov Regime-Switching Forward-Backward Stochastic Differential Equations with Jumps and Partial Information }

\author{Olivier Menoukeu-Pamen }
\address{ Institute for Financial and Actuarial Mathematics, Department of Mathematics, University of Liverpool,
United Kingdom.}
             \email{Menoukeu@liverpool.ac.uk}
\thanks{}

\date{ October 2014}

 \subjclass[2010]{93E30, 91G80, 91G10, 60G51, 60HXX, 91B30}

 \keywords{forward-backward stochastic differential equations, Malliavin calculus, regime switching, recursive utility maximization, stochastic maximum principle.}

\maketitle


\begin{abstract}
This paper presents three versions of maximum principle for a stochastic optimal control problem of Markov regime-switching forward-backward stochastic differential equations with jumps (FBSDEJs).  A general sufficient maximum principle for optimal control for a system driven by a Markov regime-switching forward and backward jump-diffusion model is developed. After, an equivalent maximum principle is proved. Malliavin calculus is also employed to derive a general stochastic maximum principle. The latter does not require concavity of Hamiltonian. Applications of the stochastic maximum principle to non-concave Hamiltonian and recursive utility maximization is also discussed.

\end{abstract}

 \section{Introduction}

 Optimal control problem for Markovian regime-switching model has received a lot of attention recently; See, e.g., \cite{Donel2011, DonHeu2011, LZ13, TW12, ZES2012}. One of the reasons for looking at regime switching in finance for example is that, they enable to capture exogenous macroeconomic cycles against which asset prices evolve (see \cite{Hamil89}.) There are two existing approaches to solve stochastic optimal control problem in the literature: The dynamic programing and the stochastic maximum principle. As for the dynamic programming, the reader may consult \cite{FS06, YZ99} and references therein.

The stochastic maximum principle is a generalization of the Pontryagin maximum principle, where optimizing a value function is turned into optimizing a functional called Hamiltonian. The stochastic maximum principle is given in terms of an adjoint equation, which is solution to a backward stochastic differential equation (BSDE). There is a vast literature on stochastic maximum principle and the reader may consult \cite{Ben83, Bis78, Kush72, OS071, Pen90, YZ99} for more information. One of the common application of stochastic maximum principle in finance is the mean-variance portfolio selection problem, which can be seen as a linear-quadratic problem; See, e.g., \cite{OS071, YZ99} and references therein. Another application of the maximum principle pertains to the utility maximization (classical and recusirve) or risk minimization; See, e.g., \cite{KPQ0, OS071, OS091}.

One of the motivations of this paper is the problem of stochastic differential utility (SDU) maximization of terminal wealth under Markov switching. The notion of recursive utility (or SDU) was introduced in \cite{DE} as a generalization of standard utility. The cost function of such utility is given in terms of an intermediate consumption rate and a future utility, therefore it can be represented as a solution of a backward stochastic differential equation (BSDE). They are many papers dealing with SDU maximization; See e.g., \cite{KPQ0} and references therein.

Stochastic maximum principle for regime switching models was introduced in \cite{Donel2011, DonHeu2011} for Markov regime-switching diffusion systems and extended in \cite{ZES2012} for Markov regime-switching jump-diffusion systems. In both cases, the authors developed a sufficient stochastic maximum principle. However, when solving the sufficient maximum principle, one of the main assumption is the concavity. Furthermore, in many applications, the concavity assumption may be violated. In \cite{LZ13}, the authors prove a weak sufficient and necessary maximum principle (that does not require concavity assumption) for Markov regime-switching diffusion systems. In this paper, we are able to solve an optimal control problem with non concave utility function for Markov regime-switching jumps-diffusion based on Malliavin calculus.

This paper discusses a partial information stochastic maximum principle for optimal control of forward backward stochastic differential equation (FBSDE) driven by Markov regime-switching jump-diffusion process. We first prove a general sufficient maximum for optimal control with partial information (Theorem \ref{mainressuf1}). This can be seen as a generalization of \cite[Theorem 3.1]{ZES2012} to the FBSDE setting, and of \cite[Theorem 2.3]{OS091} to the regime-switching setting. Second, we prove a version of a stochastic maximum principle which does not require concavity condition. The latter version can be seen as an equivalent maximum principle. In fact, a critical point for the performance functional of a partial information FBSDE problem is a conditional critical point for the associated Hamiltonian and vice versa. The proof of such equivalent maximum principle requires the use of some variational equations (compare with \cite[Section 4]{Pen93}). Note that the result obtained in this case is of a local form. This result is an extension of \cite[Theorem 3.1]{OS091} to the regime-switching setting. One of the main drawback of the two preceding maximum principles is the use of the adjoint processes which are defined in terms of backward stochastic differential equation (BSDE). These equations are usually hard to solve explicitly. The Malliavin calculus approach is then used to overcome this problem. This approach was introduced in \cite{MOZ12} and further developed in \cite{DOPP2011, MMPS13}. In this set up, the adjoint processes are replaced by other processes given in terms of the coefficients of the system and not by a BSDE. Note also that, the concavity condition is not needed in this approach. Using the Malliavin calculus approach, the results obtained in \cite[Example 4.7]{LZ13} can be extended to the jump-diffusion case. The results given here also generalized the ones derived in \cite{TW12}. We also show that our result can be applied to a problem stochastic differential utility (SDU) maximization of terminal wealth under Markov switching.

The paper is organized as follows: In Section \ref{framew}, the framework for the partial information control problem is introduced. Section \ref{maxiprinc1} presents a partial information sufficient maximum principle for forward backward stochastic differential equation (FBSDE) driven by Markov switching jump-diffusion process. An equivalent maximum principle is also given. In Section \ref{mallcalappro}, we a Malliavin calculus approach to solve the control problem. Section \ref{appli1} uses the results obtained to solve a problem of optimal control for Markov switching jump-diffusion model. A problem of recursive utility maximization with Markovian regime-switching is studied.

 \section{Framework}\label{framew}

This section presents the model and formulates the stochastic control problem in a continuous-time Markov regime-switching forward-backward stochastic differential equations with jumps. Here, the model in \cite{ZES2012} shall be adopted for the forward Markov regime-switching jump-diffusion model. Let $\{B(t)\}_{0\leq t\leq T}$ be a Brownian motion on the filtered probability space $
(\Omega ^{(B)},\mathcal{F}^{(B)},\{\mathcal{F}_{t}^{(B)}\mathcal{\}}_{0\leq
t\leq T},P^{(B)}),$ where $\{\mathcal{F}^{B}\mathcal{\}}_{0\leq t\leq T}$ is the $P^{(B)}$-augmented filtration generated by $B(t)$ with $\mathcal{F}^{(B)}=\mathcal{F}%
_{T}^{(B)}.$ Analogously, let a stochastic basis $(\Omega ^{(\widetilde{N})},\mathcal{F}^{(\widetilde{N})},\{\mathcal{F}_{t}^{(\widetilde{N})}\mathcal{\}}_{0\leq
t\leq T},P^{(\widetilde{N})}),$ associated with the compensated Poisson random measure $\widetilde{N}%
(\diffns t, \diffns z):=N(\diffns \zeta,\diffns s)-\nu (\diffns \zeta)\diff s. $ with L\'{e}vy measure $\nu $.
 and let a stochastic basis $(\Omega ^{(\widetilde{\Phi})},\mathcal{F}^{(\widetilde{\Phi})},\{\mathcal{F}_{t}^{(\widetilde{\Phi})}\mathcal{\}}_{0\leq
t\leq T},P^{(\widetilde{\Phi})}),$ associated with the martingale generated by a continuous-time, finite-state, observable Markov chain $\{\alpha(t)\}_{0\leq t\leq T}$ (see \eqref{semi-mar1}-\eqref{semi-mar3}).

In the following, we shall confine ourselves to the stochastic basis $(\Omega ,\mathcal{F},\mathbb{F}=\{\mathcal{F}_{t}\mathcal{\}}_{0\leq t\leq T},P)$,
where $\Omega =\Omega ^{(B)}\times \Omega ^{(\widetilde{N})} \times \Omega ^{(\alpha)},$ $\mathcal{F}=\mathcal{F%
}^{(B)}\times \mathcal{F}^{(\widetilde{N})}\times \mathcal{F}^{(\widetilde{\Phi})}$, $\mathcal{F}_{t}=\mathcal{F}%
_{t}^{(B)}\times \mathcal{F}_{t}^{(\widetilde{N})}\times \mathcal{F}_{t}^{(\widetilde{\Phi})}$, \\$P=P^{(B)}\times P^{(\widetilde{N})}\times P^{(\widetilde{\Phi})}.$

 $\alpha:=\{\alpha(t)\}_{0\leq t\leq T}$ is an irreducible homogeneous continuous-time Markov chain with a finite state space $\mathbb{S}=\{\mathbf{e}_1,\mathbf{e}_2, \ldots ,\mathbf{e}_D\}\subset \mathbb{R}^D$, where $D\in \mathbb{N}$, and the $j$th component of $e_i$ is the Kronecker delta $\delta_{ij}$ for each $i,j=1,\ldots, D$.  The Markov chain is characterized by a rate (or intensity) matrix $\Lambda:=\{\lambda_{ij}:1\leq i,j\leq D\}$ under $P$. Note that, for each $1\leq i,j\leq D,\,\,\lambda_{ij}$ is the transition intensity of the chain from state $e_i$ to state $e_j$ at time $t$. Recall that for $i\neq j,\,\,\lambda_{ij}\geq 0$ and $\sum_{j=1}^D \lambda_{ij}=0$, hence $\lambda_{ii}\leq 0$.

 It follows from \cite{EAM94} that $\alpha$  admits the following semimartingale representation
\begin{align}\label{semi-mar1}
\alpha(t)=\alpha(0)+\int_0^t\Lambda ^T\alpha(s)\diffns s+M(t),
\end{align}
where $M:=\{M(t)\}_{t\in [0,T]}$ is a $\mathbb{R}^D$-valued $(\mathbb{F},P)$-martingale and $\Lambda ^T$ denotes the transpose of a matrix.

We shall introduce the set of jump martingales associated with the Markov chain. For each $1\leq i,j\leq D$, with $i\neq j$, and $t\in [0,T]$, denote by $J^{ij}(t)$ the number of jumps from state $e_i$ to state $e_j$ up to time $t$. It can be shown (see \cite{EAM94}) that
\begin{align}\label{semi-mar2}
J^{ij}(t)=\lambda_{ij}\int_0^t\langle \alpha(s-),e_i\rangle \diffns s +m_{ij}(t),
\end{align}
where $m_{ij}:=\{m_{ij}(t)\}_{t\in [0,T]}$ with $m_{ij}:=\int_0^t\langle \alpha(s-),e_i\rangle \langle \diffns M(s),e_j\rangle$ is a $(\mathbb{F},P)$-martingale.

Fix $j\in \{1,2,\ldots,D\}$, denote by $\Phi_j(t)$ the number of jumps into state $e_j$ up to time $t$. Then
\begin{align}\label{semi-mar3}
\Phi_j(t)&:=\sum_{i=1,i\neq j}^D J^{ij}(t)= \sum_{i=1,i\neq j}^D \lambda_{ij}\int_0^t\langle \alpha(s-),e_i\rangle \diffns s +\widetilde{\Phi}_{j}(t)\notag\\
&= \lambda_j(t) + \widetilde{\Phi}_{j}(t),
\end{align}
with $\widetilde{\Phi}_{j}(t)=\sum_{i=1,i\neq j}^D m_{ij}(t)$ and $\lambda_j(t)=\sum_{i=1,i\neq j}^D \lambda_{ij}\int_0^t\langle \alpha(s-),e_i\rangle \diffns s
$. Note that, for each
 $j\in \{1,2,\ldots,D\},\,\,\,\widetilde{\Phi}_{j}:=\{\widetilde{\Phi}_{j}(t)\}_{t\in [0,T]}$ is a $(\mathbb{F},P)$-martingale.

Let introduce a Markov regime-switching Poisson random measure. Assume that  $N(\diffns \zeta,\diffns s)$ is a Poisson random measure on $\Big(\mathbb{R}_+\times \mathbb{R}_0,\mathcal{B}(\mathbb{R}_+) \otimes\mathcal{B} _0\Big)$ where $\mathbb{R}_{0}:=\mathbb{R}\backslash \left\{ 0\right\}$ and $\mathbb{R}_{+}:=[0,+\infty)$. Denote by $ \nu_\alpha(\diffns \zeta)\diffns t  $ its compensator (or dual predictable projection), then it is defined by:
\begin{align}\label{comp1}
\nu_\alpha(\diffns \zeta)\diffns t:=\sum_{j=1}^D\langle \alpha(t-),e_j\rangle \nu_j(\diffns \zeta)\diffns t.
\end{align}
For each $j\in \{1,2,\ldots,D\},\,\,\nu_j(\diffns \zeta)$ is the conditional density of the jump size when the Markov chain $\alpha$ is in state $e_i$ and satisfies $\int_{\mathbb{R}_{0}}\min(1,\zeta^{2})\nu_i (\diffns \zeta)< \infty$. Moreover, define $\widetilde{N}_\alpha(\diffns \zeta,\diffns s)$ by
\begin{equation}\label{comppoissrm1}
\widetilde{N}_\alpha(\diffns \zeta,\diffns s):=N(\diffns \zeta,\diffns s)-\nu_\alpha(\diffns \zeta)\diffns t
\end{equation}

Suppose that the state process $X(t)=X^{(u)}(t,\omega);\,\,0 \leq t \leq T,\,\omega \in \Omega$ is a controlled Markov regime-switching jump-diffusion of the form
\begin{equation}\left\{
\begin{array}{llll}
\diff X (t) & = & b(t,X(t),\alpha(t),u(t),\omega)\diff t +\sigma (t,X(t),\alpha(t),u(t),\omega)\diff B(t)  \\
&  & +\displaystyle \int_{\mathbb{R}_0}\gamma(t,X(t),\alpha(t),u(t),\zeta, \omega)\,\widetilde{N}_\alpha(\diffns \zeta,\diffns t)\\
& &+\eta (t,X(t),\alpha(t),u(t),\omega)\cdot \diffns \widetilde{\Phi}(t),\,\,\,\,\,\, t \in [ 0,T] \label{eqstateprocess1} \\
X(0) &=& x_0,
\end{array}\right.
\end{equation}
where $T>0$ is a given constant. $u(\cdot)$ is the control process.

The functions $b:[0,T]\times \mathbb{R} \times \mathbb{S} \times \mathcal{U} \times \Omega \rightarrow \mathbb{R}\,,\,\, \sigma:[0,T]\times \mathbb{R} \times \mathbb{S} \times \mathcal{U} \times \Omega \rightarrow \mathbb{R}$, $\gamma:[0,T]\times \mathbb{R} \times \mathbb{S} \times \mathcal{U} \times \mathbb{R}_0\times \Omega \rightarrow \mathbb{R}$ and $\eta:[0,T]\times \mathbb{R} \times \mathbb{S} \times \mathcal{U} \times \Omega \rightarrow \mathbb{R}$ are given such that for all $t,\,\,\,b(t,x,e_i,u,\cdot)$, $ \sigma(t,x,e_i,u,\cdot)$, $\gamma(t,x,e_i,u,z,\cdot)$ and $ \eta(t,x,e_i,u,\cdot)$ are $\mathcal{F}_t$-measurable for all $x \in\mathbb{R},\,\,\,e_i \in\mathbb{S},\,\,u \in\mathcal{U}$ and $z \in\mathbb{R}_0$.

 We suppose that we are given a subfiltration
\begin{align}\label{subfil}
\mathcal{E}_t\subset \mathcal{F}_t\,;\,\,\,t\in [0,T],
\end{align}
representing the information available to the controller at time $t$. Note that one possible subfiltration $\mathcal{E}_{t}$ in \eqref{subfil} is the $%
\delta $-delayed information given by $
\mathcal{E}_{t}=\mathcal{F}_{(t-\delta)^+ };\,\,\,t\geq 0,$ where $\delta \geq 0$ is a given constant delay.

We consider the associated BSDE's in the unknowns  $\Big(Y(t), Z(t), K(t,\zeta), V(t)\Big)$ of the form
\begin{equation}\left\{
\begin{array}{llll}
\diff Y (t) & = & - g(t,X(t),\alpha(t),Y(t), Z(t), K(t,\cdot),V(t),u(t))\diff t \,+\,Z(t)\diff B(t)  \\
&  & + \displaystyle \int_{\mathbb{R}_0}K(t,\zeta)\,\widetilde{N}_\alpha(\diffns \zeta,\diffns t)+V(t)\cdot \diffns \widetilde{\Phi}(t);\,\,\, t \in [ 0,T]  \label{eqassBSDE} \\
Y(T) &=& h(X(T),\alpha(T))\,,
\end{array}\right.
\end{equation}
where $g:[0,T]\times \mathbb{R} \times \mathbb{S} \times \mathbb{R} \times \mathbb{R} \times \mathcal{R} \times \mathbb{R} \times \mathcal{U} \times \Omega \rightarrow \mathbb{R}$ and $h:\mathbb{R}\times \mathbb{S}  \rightarrow \mathbb{R}$ are such that the BSDE \eqref{eqassBSDE} has a unique solution. As for sufficient conditions for existence and uniqueness of Markov regime-switching BSDEs, we refer the reader for e.g., to \cite{CoEl10} or \cite{Crep} and references therein.

Let $f:[0,T]\times \mathbb{R} \times \mathbb{S}  \times \mathbb{R} \times \mathbb{R}\times \mathcal{R}  \times \mathbb{R} \times \mathcal{U} \times \Omega \rightarrow \mathbb{R}, \,\,\,\varphi:\mathbb{R}  \times \mathbb{S}\rightarrow \mathbb{R}$ and $\psi:\mathbb{R}  \rightarrow \mathbb{R}$ be given $C^1$ functions with respect to their arguments. Assume that the performance functional is as follows
\begin{align}
J(u):=E\Big[ \int_0^T  f(s,X(s),\alpha(s),Y(s), Z(s), K(s,\cdot),V(s),u(s))\diff s + \varphi(X(T),\alpha(T))\,+\,\psi(Y(0))\Big].\label{perfunct1}
\end{align}
Here, $f,\,\varphi$ and $\psi$ may be seen as profit rates, bequest functions and ``utility evaluations" respectively, of the controller.

Let $\mathcal{A}_{\mathcal{E}}$ denote the family of admissible control $u$, such that there are contained in the set of $\mathcal{E}_t$-predictable control, and the system \eqref{eqstateprocess1}-\eqref{eqassBSDE} has a unique solution, and
\begin{align*}
&E\Big[ \int_0^T\Big\{| f(t,X(t),\alpha(t),Y(t), Z(t), K(t,\cdot),V(t),u(t))|\\
&\,\,+\Big| \frac{\partial f}{\partial x_i}(t,X(t),,\alpha(t),Y(t), Z(t), K(t,\cdot),V(t),u(t)) \Big|^2  \Big\}\diffns t  \Big.\\
&\Big. \varphi(X(T),\alpha(T))+|\varphi^\prime(X(T),\alpha(T))|^2+|\psi(Y(0))|+|\psi^\prime(Y(0))|^2\Big]<\infty \text{ for } x_i=x,y,z,k \text{ and } u.
\end{align*}
The set $\mathcal{U}\subset \mathbb{R}$ is a given convex set such that $u(t)\in \mathcal{U}$ for all $t\in [0,T]$ a.s., for all $u \in \mathcal{A}_{\mathcal{E}}$.
\begin{rem}
The system \eqref{eqstateprocess1}-\eqref{eqassBSDE} is a semi-couple forward-backward SDE. Existence and uniqueness results of the SDE \eqref{eqstateprocess1} follows from existing literature under global Lipschitz continuity and growth condition of the coefficients. Therefore, existence and uniqueness of the solution of \eqref{eqstateprocess1}-\eqref{eqassBSDE} will follow from the existence and uniqueness of the BSDE \eqref{eqassBSDE}. As for existence and uniqueness of BSDE with poisson jump and Markov chain, the reader may consult \cite{CoEl10} or \cite{Crep} and references therein.
\end{rem}

The problem we consider is the following: find $u^\ast \in \mathcal{A}_{\mathcal{E}} $ such that
\begin{align}\label{prob111}
   J(u^\ast)=\sup_{u\in  \mathcal{A}_{\mathcal{E} }}J(u).
  \end{align}

\section{Maximum Principle for a Markov regime-switching Forward-Backward stochastic differential equation with jumps}\label{maxiprinc1}

In this section, we derive a general sufficient stochastic maximum principle for a forward-backward Markov regime-switching jump-diffusion model. After we shall derive an equivalent maximum principle.

For this purposes, define the Hamiltonian
\begin{equation*}
    H:[0,T] \times \mathbb{R}\times \mathbb{S}\times \mathbb{R}\times \mathbb{R}\times \mathcal{R}\times \mathbb{R}\times \mathcal{U}\times \mathbb{R}  \times \mathbb{R}\times \mathbb{R} \times \mathcal{R} \times \mathbb{R}
    \longrightarrow \mathbb{R},
\end{equation*}
by
\begin{align}
   & H\left(t,x,e_i,y,z,k,v,u,a,p,q,r(\cdot),w\right)\notag\\
    :=& f(t,x,e_i,y,z,k,v,u)+a g(t,x,e_i,y,z,k,v,u)+  p b(t,x,e_i,u) \,\notag\\
    &+q\sigma (t,x,e_i,u)+\int_{\mathbb{R}_0}r(t,\zeta)\gamma(t,x,e_i,u,\zeta)\nu_i(\diffns \zeta)+\sum_{j=1}^D\eta^j(t,x,e_i,u)w^j(t)\lambda_{ij} ,\label{Hamiltoniansddg1}
\end{align}
where $\mathcal{R} $ denotes the set of all functions $k:[0,T]\times \mathbb{R}_0 \rightarrow \mathbb{R}$ for which the integral in \eqref{Hamiltoniansddg1} converges.

We suppose that $H$ is  Fr\'echet differentiable in the variables $x,y,z,k,v,u$ and that $\nabla_k H(t,\zeta)$ is a random measure which is absolutely continuous with respect to $\nu_\alpha$. Define the adjoint processes $A(t),\,p(t),\,q(t), r(t,\cdot)$ and $w(t),\,\,\,t\in[0,T]$ associated to these Hamiltonians by the following system of Markov regime-switching FBSDEJs\begin{enumerate}
\item Forward SDE in $A(t)$
\begin{equation}\left\{
\begin{array}{llll}
\diffns  A (t) & = & \dfrac{\partial H}{\partial y} (t) \diff t + \dfrac{\partial H}{\partial z} (t) \diffns B(t)+ \displaystyle \int_{\mathbb{R}_0} \dfrac{\diffns \nabla_k H}{\diffns \nu_\alpha(\zeta)} (t,\zeta) \,\widetilde{N}_\alpha(\diffns \zeta,\diffns t) + \nabla_vH(t)\cdot\diffns \widetilde{\Phi}(t);\,\,\,t\in [0,T] \\
A(0) &=& \psi^\prime(Y(0)).\label{lambda1}
\end{array}\right.
\end{equation}
Here and in what follows, we use the notation
$$
 \dfrac{\partial H}{\partial y} (t) = \dfrac{\partial H}{\partial y} (t,X(t),\alpha(t),u(t),Y(t), Z(t), K(t,\cdot),V(t),A(t),p(t),q(t),r(t,\cdot),w(t)),
$$
etc,$\dfrac{\diffns \nabla_k H}{\diffns \nu(\zeta)} (t,\zeta) $ is the Radon-Nikodym derivative of $ \nabla_k H(t,\zeta)$ with respect to $\nu(\zeta)$ and $\nabla_vH(t)\cdot\diffns \widetilde{\Phi}(t)=\sum_{j=1}^D \dfrac{\partial H}{\partial v^j} (t)\diffns \widetilde{\Phi}_j(t)$ with $V^j=V(t,e_j)$.

\item The Markovian regime-switching BSDE in $(p(t),q(t),r(t,\cdot),w(t))$
\begin{equation}\left\{
\begin{array}{llll}
\diffns p (t) & = & - \dfrac{\partial H}{\partial x} (t) \diffns t + q(t)\diff B(t)+ \displaystyle  \int_{\mathbb{R}_0} r (t,\zeta) \,\widetilde{N}_\alpha(\diffns\zeta,\diffns t)+ w(t)\cdot\diffns \widetilde{\Phi}(t);\,\,\,t\in [0,T]  \\
p (T) &=& \dfrac{\partial \varphi}{\partial x}(X(T),\alpha(T))\,+A(T) \dfrac{\partial h}{\partial x} (X(T),\alpha(T))\label{ABSDE1},
\end{array}\right.
\end{equation}

\end{enumerate}

\subsection{A sufficient maximum principle}
In what follows, we give the sufficient maximum principle.
\begin{thm}[Sufficient maximum principle]\label{mainressuf1}
Let $\widehat{u}\in \mathcal{A}_{\mathcal{E} }$ with corresponding solutions $\widehat{X}(t),(\widehat{Y}(t), \widehat{Z}(t), \widehat{K}(t,\zeta), \widehat{V}(t)), \widehat{A}(t),(\widehat{p}(t),\widehat{q}(t),\widehat{r}(t,\zeta),\widehat{w}(t))$ of \textup{\eqref{eqstateprocess1}}, \textup{\eqref{eqassBSDE}}, \textup{\eqref{lambda1}} and \textup{\eqref{ABSDE1}} respectively. Suppose that the following are true:
\begin{enumerate}
\item The functions
\begin{align}\label{thsuffcond1}
x \mapsto  h(x, e_i),\,\,x \mapsto  \varphi(x, e_i),\,\,y \mapsto  \psi(y),
\end{align} are concave for all $t\in [0,T]$.

\item  The function

\begin{align}\label{thsuffcond2}
\widetilde{H}(x,y,z,k,v)=\essup_{u\in \mathcal{U} }E\Big[H( t,x,e_i,y,z,k,v,u,\widehat{a},\widehat{p}(t),\widehat{q}(t),\widehat{r}(t,\cdot),\widehat{w}(t))| \mathcal{E}_t\Big]
\end{align}
is concave  for all $(t,e_i) \in  [0,T]\times \mathbb{S}$ a.s.

\item
\begin{align}\label{thsuffcond4}
\underset{u\in \mathcal{U} }{\essup}&\Big\{E\Big[H (t,\widehat{X}(t),\alpha(t), u,\widehat{Y}(t), \widehat{Z}(t), \widehat{K}(t,\cdot),\widehat{V}(t),\widehat{A}(t),\widehat{p}(t),\widehat{q}(t),\widehat{r}(t,\cdot),\widehat{w}(t)) \Big. \Big |  \mathcal{E}_t\Big]\Big\} \notag\\
&=E\Big[H (t,\widehat{X}(t),\alpha(t), \widehat{u},\widehat{Y}(t), \widehat{Z}(t), \widehat{K}(t,\cdot),\widehat{V}(t),\widehat{A}(t),\widehat{p}(t),\widehat{q}(t),\widehat{r}(t,\cdot),\widehat{w}(t)) \Big. \Big |  \mathcal{E}_t\Big]
\end{align}
for all $t\in [0,T]$, a.s.

\item Assume that  $\frac{\diffns }{\diffns \nu}\nabla_k\widehat{g}(t,\xi)>-1$.
\item In addition, assume the following growth condition
\begin{align}\label{thsuffcond5}
&E\Big[\int_0^T\Big\{  \widehat{p}^2(t)  \Big(  (\sigma(t)-\widehat{\sigma}(t))^2+ {\int}_{\mathbb{R}_0}( \gamma(t,\zeta)-\widehat{\gamma} (t,\zeta) )^2\,\nu_\alpha(\diffns \zeta)+\sum_{j=1}^D(\eta^j(t)-\widehat{\eta}^j(t) )^2\lambda_{j}(t) \Big)\Big.    \Big.\notag\\
&+(X(t)-\widehat{X}(t))^2 \Big(  \widehat{q}^2(t)+  {\int}_{\mathbb{R}_0}\widehat{r}^2 (t,\zeta) \nu_\alpha(\diffns  \zeta)+\sum_{j=1}^D(w^j)^2(t)\lambda_{j}(t)  \Big)\notag\\
&+(Y(t)-\widehat{Y}(t))^2 \Big( (\dfrac{\partial \widehat{H}} {\partial z} )^2(t)
+  {\int}_{\mathbb{R}_0} \Big\| \nabla_k \widehat{H}(t,\zeta)\Big\|^2 \nu_\alpha(\diffns \zeta)+ \sum_{j=1}^D  (\dfrac{\partial \widehat{H}} {\partial v^j} )^2(t) \lambda_{j}(t)  \Big)\notag\\
&\Big.    \Big. +\widehat{A}^2(t)  \Big(  (Z(t)-\widehat{Z}(t))^2+ {\int}_{\mathbb{R}_0}( K (t,\zeta)-\widehat{K} (t,\zeta) )^2\nu_\alpha(\diffns  \zeta) + \sum_{j=1}^D(V^j(t)-\widehat{V}^j(t) )^2\lambda_{j}(t)   \Big)\Big\} \diffns t \Big]<\infty .
\end{align}
\end{enumerate}
Then $\widehat{u}$ is an optimal control process and $\widehat{X}$ is the corresponding controlled state process.

\end{thm}

\begin{rem}
In \textup{Theorem \ref{mainressuf1}} and in the following, we shall use the notations
$X(t)=X^{\widehat{u}}(t)$ and $Y(t)=Y^{\widehat{u}}(t)$ are the processes associated to the control $\widehat{u}(t)$.
Furthermore, put\\
$
\dfrac{\partial \widehat{H}}{\partial x}(t):=\dfrac{\partial H}{\partial x}  H (t,\widehat{X}(t),\alpha(t), \widehat{u},\widehat{Y}(t), \widehat{Z}(t), \widehat{K}(t,\cdot),\widehat{V}(t),\widehat{A}(t),\widehat{p}(t),\widehat{q}(t),\widehat{r}(t,\cdot),\widehat{w}(t)) $
and similarly for $
\dfrac{\partial \widehat{H}}{\partial y}(t), \dfrac{\partial \widehat{H}}{\partial z}(t), \nabla_{k} \widehat{H}(t,\zeta), \dfrac{\partial \widehat{H}}{\partial v^j}(t)$ and $\dfrac{\partial \widehat{H}}{\partial u}(t)$.

\end{rem}

\begin{proof}[Proof of Theorem \protect\ref{mainressuf1}]
We shall prove that
$J(x,\widehat{u},e_i)\geq J(x,u,e_i) \text{ for all } u \in \mathcal{A}_{\mathcal{E}}.$

Choose  $u \in  \mathcal{A}_{\mathcal{E}}$ and consider
\begin{align}\label{eqdiff1}
J(x,u,e_i)-J(x,\widehat{u},e_i)=I_1+I_2+I_3,
\end{align}
where
\begin{align}
I_1=&E\Big[ \int_0^T \Big\{ f(t,X(t),\alpha(t),Y(t), Z(t), K(t,\cdot),V(t),u(t))\notag\\
&-f(t, \widehat{X}(t),\alpha(t), \widehat{Y}(t),  \widehat{Z}(t),  \widehat{K}(t,\cdot), \widehat{V}(t), \widehat{u}(t)) \Big\}\diff t \Big], \label{eqI11}\\
I_2=&E\Big[ \varphi(X(T),\alpha(T)) -\varphi(\widehat{X}(T), \alpha(T))\Big],\label{eqI21}\\
I_3=&E\Big[\psi(Y(0))\,-\,\psi(\widehat{Y}(0))\Big]. \label{eqI31}
\end{align}
By the definition of $H$, we get
\begin{align}
I_1=&E\Big[ \int_0^T \Big\{ H(t)-\widehat{H}(t)- \widehat{A}(t)(g(t)-\widehat{g}(t))-\widehat{p}(t)(b(t)-\widehat{b}(t))-\widehat{q}(t)(\sigma(t)-\widehat{\sigma}(t))\Big.\Big.\notag\\
&\Big.\Big. -\int_{\mathbb{R}_0}\widehat{r}(t,\zeta)(\gamma(t,\zeta)-\widehat{\gamma}(t,\zeta))\nu_\alpha(\diff\zeta)   -\sum_{j=1}^D\widehat{w}^j(t)(\eta^j(t)-\widehat{\eta}^j(t) )\lambda_{j}(t)   \Big\} \diff t\Big]. \label{eqI12}
\end{align}
By the concavity of $\varphi$ in $x$, the It\^o formula, \eqref{eqstateprocess1}, \eqref{ABSDE1} and \eqref{thsuffcond5} we get
\begin{align}\label{eqI22}
I_2\leq& E\Big[\dfrac{\partial \varphi}{\partial x} (\widehat{X}(T),\alpha(T))(X(T)-\widehat{X}(T)) \Big]\notag\\
=& E\Big[\widehat{p}(T) (X(T)-\widehat{X}(T))\Big]-E\Big[\widehat{A}(T)\dfrac{\partial h}{\partial x}(\widehat{X}(T),\alpha(T))(X(T)-\widehat{X}(T))\Big]\notag\\
=& E\Big[\int_0^T\Big\{\widehat{p}(t) (b(t)-\widehat{b}(t))\diff t+(X(t^-)-\widehat{X}(t^-))
  \Big(-\frac{\partial \widehat{H}}{\partial x}(t)\Big) +(\sigma(t)-\widehat{\sigma}(t))\widehat{q}(t) \Big.\notag\\
&+\int_{\mathbb{R}_0}(\gamma(t,\zeta)-\widehat{\gamma}(t,\zeta))\widehat{r}(t,\zeta)\nu_\alpha(\diffns\zeta)+  \sum_{j=1}^D\widehat{w}^j(t)(\eta^j(t)-\widehat{\eta}^j(t) )\lambda_{j}(t)   \Big\} \diff t    \Big]\notag\\
&-E\Big[\widehat{A}(T)\dfrac{\partial h}{\partial x} (\widehat{X}(T),\alpha(T))(X(T)-\widehat{X}(T))\Big].
\end{align}

By the concavity of $\psi,h$, the It\^o formula, \eqref{eqassBSDE} and \eqref{lambda1}, we get
\begin{align}\label{eqI32}
I_3\leq & E\Big[\psi^\prime(\widehat{Y}(0))(Y(0)-\widehat{Y}(0))\Big]\notag\\
=& E\Big[\widehat{A}(0)(Y(0)-\widehat{Y}(0))\Big]\notag\\
=& E\Big[\widehat{A}(T) \{h(X(T),\alpha(T))-h(\widehat{X}(T),\alpha(T))\} \Big]-E\Big[\int_0^T \Big\{\dfrac{\partial \widehat{H}}{\partial y}(t) (Y(t)-\widehat{Y}(t))  \Big.\notag\\
+&\widehat{A}(t) (-g(t)+\widehat{g}(t)) + (Z(t)-\widehat{Z}(t))\dfrac{\partial \widehat{H}}{\partial z}(t) \notag\\
 \Big.+&\int_{\mathbb{R}_0}(K(t,\zeta)-\widehat{K}(t,\zeta))\nabla_k \widehat{H}(t,\zeta)\nu_\alpha(d\zeta)+\sum_{j=1}^D\dfrac{\partial \widehat{H}}{\partial v^j}(t)  (V^j(t)-\widehat{V}^j(t) )\lambda_{j}(t)  \Big\}  \diff t  \Big] \notag
\end{align}
\begin{align}
\leq& E\Big[\widehat{A}(T)\dfrac{\partial h}{\partial x}( \widehat{X}(T),\alpha(T))(X(T)-\widehat{X}(T)) \Big]-E\Big[\int_0^T\Big\{\dfrac{\partial \widehat{H}}{\partial y}(t) (Y(t)-\widehat{Y}(t)) \Big.\notag\\
+&\widehat{A}(t) (-g(t)+\widehat{g}(t)) + (Z(t)-\widehat{Z}(t))\dfrac{\partial \widehat{H}}{\partial z}(t) \notag\\
\Big.+&\int_{\mathbb{R}_0}(K(t,\zeta)-\widehat{K}(t,\zeta))\nabla_k \widehat{H}(t,\zeta)\nu_\alpha(d\zeta) +  \sum_{j=1}^D\dfrac{\partial \widehat{H}}{\partial v^j}(t)  (V^j(t)-\widehat{V}^j(t) )\lambda_{j}(t)  \Big\}   \diff t  \Big].
\end{align}
Summing \eqref{eqI12}-\eqref{eqI32} up, we have
\begin{align}
I_1+I_2+I_3\leq & E\Big[ \int_0^T \Big\{H(t)- \widehat{H}(t)-\dfrac{\partial \widehat{H}}{\partial x}(t)(X(t)-\widehat{X}(t)) -\dfrac{\partial \widehat{H}}{\partial y}(t)(Y(t)-\widehat{Y}(t)) \Big. \Big. \notag\\
&\Big. \Big. +\int_{\mathbb{R}_0}(K(t,\zeta)-\widehat{K}(t,\zeta))\nabla_k \widehat{H}(t,\zeta)\nu_\alpha(d\zeta)\notag\\
&+ \sum_{j=1}^D\dfrac{\partial \widehat{H}}{\partial v^j}(t)  (V^j(t)-\widehat{V}^j(t) )\lambda_{j}(t) \Big\} \diffns t\Big].\label{eqdiff2}
\end{align}
One can show, using the same arguments in \cite{FOS05} (see also \cite{ZES2012}) that, the right hand side of \eqref{eqdiff2} is non-positive.  This completed the proof.
\end{proof}

\subsection{An equivalent maximum principle}

In this section, we shall show a version of maximum principle which does not require concavity condition. We shall call it an equivalent maximum principle. Let us make the following assumptions

\begin{asum}\label{assumces1}
For all $t_0\in [0,T]$ and all bounded $\mathcal{E}_t$-measurable random variable $\theta(\omega)$, the control process $\beta(t)$ defined by
\begin{align}\label{eqbeta1}
\beta(t):= \chi_{]t_0,T[}(t)\theta(\omega);\,\,t\in [0,T], \text{ belongs to } \mathcal{A}_{\mathcal{E}},
\end{align}
\end{asum}

\begin{asum}\label{assumces2}
For all $u \in \mathcal{A}_{\mathcal{E}}$ and all bounded $\beta \in \mathcal{A}_{\mathcal{E}}$, there exists $\delta>0$ such that
\begin{align}\label{eqbeta2}
\widetilde{u}(t):=u(t)+\ell \beta(t) \in \mathcal{A_{\mathcal{E}}};\,\,t\in [0,T] , \text{ belongs to } \mathcal{A}_{\mathcal{E}} \text{ for all } \ell \in ]-\delta,\delta[.
\end{align}
\end{asum}

\begin{asum}\label{assumces3}
For all bounded $\beta \in \mathcal{A}_{\mathcal{E}}$, the derivatives processes
\begin{align*}
x_1(t)=\dfrac{\diffns}{\diffns \ell }X^{(u+\ell \beta)}(t)\Big. \Big|_{\ell=0};& \,\,\,\,y_1(t)=\dfrac{\diffns}{\diffns \ell }Y^{(u+\ell \beta)}(t)\Big. \Big|_{\ell=0}
;\\
z_1(t)=\dfrac{\diffns}{\diffns \ell }Z^{(u+\ell \beta)}(t)\Big. \Big|_{\ell=0}; \,\,\,\,&k_1(t)=\dfrac{\diffns}{\diffns \ell }K^{(u+\ell \beta)}(t,\cdot)\Big. \Big|_{\ell=0};\\
\,\,\,\,\,\,\,\,v_1^j(t)=\dfrac{\diffns}{\diffns \ell }V^{j,(u+\ell \beta)}(t)\Big. \Big|_{\ell=0},\,\,\,j=1,\ldots, D&
\end{align*}
exist and belong to $L^2(\lambda \times P)$.
\end{asum}
In the following, we write $\dfrac{\partial b}{\partial x}(t)$ for $\dfrac{\partial b}{\partial x}(t,X(t),\alpha(t),u(t))$, etc.  It follows from \eqref{eqstateprocess1} and \eqref{eqassBSDE} that

\begin{equation}\left\{
\begin{array}{llll}
\diffns x_1(t) & =& \Big\{\dfrac{\partial b}{\partial x}(t)x_1(t) +\dfrac{\partial b}{\partial u}(t)\beta(t)\Big\}\diffns t
+\Big\{x_1(t)\dfrac{\partial \sigma}{\partial x}(t)+ \dfrac{\partial \sigma}{\partial u}(t)\beta(t)\Big\}\diffns B(t)\\
&& +\displaystyle  \int_{\mathbb{R}_0}\Big\{ \dfrac{\partial \gamma}{\partial x}(t,\zeta) x_1(t)+ \dfrac{\partial \gamma}{\partial u}(t,\zeta)\beta(t)\Big\}\widetilde{N}_\alpha(\diffns t,\diffns \zeta)\\
&& +\Big\{\dfrac{\partial \eta}{\partial x}(t) x_1(t)   +\dfrac{\partial \eta}{\partial u}(t)\beta(t)\Big\}\cdot\diffns \widetilde{\Phi}(t);\,\,t\in[0,T] \\
x_1(t)&=&0.\label{derivstate1}
\end{array}\right.
\end{equation}
and
\begin{equation}\left\{
\begin{array}{llll}
\diffns y_1(t)&=&-\Big\{\dfrac{\partial g}{\partial x}(t)x_1(t)+\dfrac{\partial g}{\partial y}(t)y_1(t)+\dfrac{\partial g}{\partial z}(t)z_1(t)+\displaystyle  \int_{\mathbb{R}_0}\nabla_k g (t)k_1(t,\zeta)\nu_\alpha(\diffns \zeta)\Big. \\
&&+\sum_{j=1}^D \dfrac{\partial g}{\partial v^j}(t)v_1^j(t)\lambda_j(t)+\dfrac{\partial g}{\partial u}(t)\beta(t)\Big\} \diffns t +z_1(t)\diff  B(t)  \\
&&+\displaystyle \int_{\mathbb{R}_0}k_1(t,\zeta) \widetilde{N}_\alpha(\diffns \zeta, \diffns t) + v_1(t)\cdot\diffns \widetilde{\Phi}(t) ;\,\,\,t\in [0,T]\\
y_1(T)&=&\dfrac{\partial h}{\partial x}(X(T),\alpha(T))x_1(T) .\label{derivassBSDE1}
\end{array}\right.
\end{equation}

\begin{rem}\label{rem111}
As for sufficient conditions for the existence and uniqueness of solutions \eqref{derivstate1} and \eqref{derivassBSDE1}, the reader may consult \cite[(4.1)]{Pen93}

As an example, a set of sufficient conditions under which \eqref{derivstate1} and \eqref{derivassBSDE1} admit a unique solution is as follows:

\begin{enumerate}
\item Assume that the coefficients $b,\sigma, \gamma, \eta, g, f, \psi$ and $\phi$ are continuous with respect to their arguments and are continuously differentiable  with respect to $(x,y,z,k,v,u)$. (Here, the dependence of $g$ and $f$ on $k$ is trough $\int_{\mathbb{R}_0}k(\zeta)\rho(t,\zeta)\nu(\diffns \zeta)$, where $\rho$ is a measurable function satisfying $0\leq \rho(t,\zeta)\leq c(1\wedge |\zeta|), \text{  } \forall \zeta\in\mathbb{R}_0$. Hence the differentiability in this argument is in the Fr\'echet sense.)

\item The derivatives of $b,\sigma, \gamma, \eta$ and $g$ are bounded.

\item The derivatives of $f$ are bounded by $C(1+|x|+|y|+(\int_{\mathbb{R}_0}|k(.,\zeta)|^2\nu(\diffns \zeta))^{1\backslash 2}+|v|+|u|)$.

\item The derivatives of $\psi$ and $\phi$ with respect to $x$ are bounded by $C(1+|x|).$
\end{enumerate}

\end{rem}

\begin{thm}[Equivalent Maximum Principle]\label{theomainneccon1}
Let $u\in \mathcal{A}_{\mathcal{E}}$ with corresponding solutions $X(t)$ of \textup{\eqref{eqstateprocess1}}, $(Y(t),Z(t),K(t,\zeta),V(t))$ of \textup{\eqref{eqassBSDE}}, $A(t)$ of \textup{\eqref{lambda1}}, $(p(t),q(t),r(t,\zeta),w(t))$ of \textup{\eqref{ABSDE1}} and corresponding derivative processes $x_1(t)$ and $(y_1(t),z_1(t),k_1(t,\zeta),v_1(t))$ given by \textup{\eqref{derivstate1}} and \textup{\eqref{derivassBSDE1}} respectively. Suppose that
\textup{Assumptions \ref{assumces1}, \ref{assumces2}} and \textup{\ref{assumces3}} hold. Moreover, assume the following growth conditions
 \begin{align}
&E\Big[\int_0^T p^2(t)\Big\{\Big(\dfrac{\partial \sigma}{\partial x}\Big)^2(t)x^2_1(t) +\Big(\dfrac{\partial \sigma}{\partial u}\Big)^2(t)\beta^2(t) +\int_{\mathbb{R}_0}\Big( \Big(\dfrac{\partial \gamma}{\partial x}\Big)^2(t,\zeta)x_1^2(t)+\Big(\dfrac{\partial \gamma}{\partial u}\Big)^2(t,\zeta)\beta^2(t)\Big)\nu_\alpha (\diffns \zeta) \Big.\notag\\
&\Big.+\sum_{j=1}^D \Big(\Big(\dfrac{\partial \eta^j}{\partial x}\Big)^2(t)x^2_1(t)+\Big(\dfrac{\partial \eta^j}{\partial u}\Big)^2(t)\beta^2(t)\Big)\lambda_j(t)\Big\} \diffns t \notag\\
&+\int_0^Tx_1^2(t)\Big\{ q^2(t)+\int_{\mathbb{R}_0}r^2(t,\zeta)\nu_\alpha(\diffns \zeta)+\sum_{j=1}^D  (\eta^j)^2(t)\lambda_j(t)\Big\}\diffns t\Big] <\infty \label{thneccond1},
\end{align}
and
\begin{align}
&E\Big[\int_0^Ty_1^2(t)  \Big\{ (\dfrac{\partial H}{\partial z})^2 (t) \,+\,\int_{\mathbb{R}_0} \|\nabla_k H\|^2(t,\zeta)\nu_\alpha (\diffns\zeta)+  \sum_{j=1}^D(\dfrac{\partial H}{\partial v^j})^2 (t) \lambda_j(t)\Big\}\diffns t\notag\\
&+ \int_0^TA^2(t)\Big\{z_1^2(t)+\int_{\mathbb{R}_0}k_1^2(t,\zeta)\nu_\alpha (\diffns \zeta)+ \sum_{j=1}^D (v^j_1)^2(t)\lambda_j(t)\Big\}\diffns t\Big]<\infty.
\label{thneccond12}
\end{align}

Then the following are equivalent:

\textup{(1)} $\dfrac{\diffns}{\diffns \ell}J^{(u+\ell \beta)}(t)\Big. \Big|_{\ell=0}=0 \text{ for all bounded } \beta\in \mathcal{A}_{\mathcal{E}}.$

\textup{(2)} $E\Big[\dfrac{\partial H}{\partial u} (t,X(t),\alpha(t),Y(t), Z(t), K(t,\cdot),V(t),u,A(t),p(t),q(t),r(t,\cdot),w(t))_{u=u(t)} \Big. \Big| \mathcal{E}_t\Big]=0$ for a.a. $ t \in [0,T].$

\end{thm}

\begin{proof}
We have that
\begin{align}
&\dfrac{\diffns }{\diffns \ell}J^{(u+\ell \beta)}(t)\Big. \Big|_{\ell=0}\notag\\
=&E\Big[\int_0^T \Big\{\dfrac{\partial f}{\partial x}(t)x_1(t)  +\dfrac{\partial f}{\partial y}(t)y_1(t)+\dfrac{\partial f}{\partial z}(t) z_1(t)+\int_{\mathbb{R}_0}\nabla_k f(t)k_1(t,\zeta)\nu_\alpha(\diffns \zeta)\Big. \label{eqIprim11}\\
&\Big.+ \sum_{j=1}^D \dfrac{\partial f}{\partial v^j}(t)v_1^j(t)\lambda_j(t)+ \dfrac{\partial f}{\partial u}(t)\beta(t)\Big\} \diffns t+\dfrac{\partial \varphi}{\partial x}  (X(T),\alpha(T))x_1(T)+\psi^\prime(Y(0))y_1(0)\Big].\notag
\end{align}
By  \eqref{ABSDE1}, the It\^o formula, \eqref{derivstate1} and \eqref{thneccond1}, we have
\begin{align}
&E\Big[\dfrac{\partial \varphi}{\partial x}(X(T),\alpha(T))x_1(T)\Big] \notag\\
=&E\Big[p(T)X(T)\Big]- E\Big[\dfrac{\partial h}{\partial x}(X(T),\alpha(T))A(T)x_1(T)\Big]   \label{eqIprim2} \\
=&E\Big[\int_0^T\Big\{p(t)\Big(\dfrac{\partial b}{\partial x}(t)x_1(t)+\dfrac{\partial b}{\partial u}(t)\beta(t)\Big)-x_1(t) \dfrac{\partial H}{\partial x}(t)\Big.\Big.\notag\\
&  +q(t)\Big(\dfrac{\partial \sigma}{\partial x}(t)x_1(t)+\dfrac{\partial \sigma}{\partial u}(t)\beta(t)\Big)+\int_{\mathbb{R}_0} r(t,\zeta)\Big(\dfrac{\partial \gamma}{\partial x}(t,\zeta)x_1(t) +\dfrac{\partial \gamma}{\partial u}(t,\zeta)\beta(t)\Big) \nu_\alpha(\diffns \zeta)\bigg.\notag\\
& +\sum_{j=1}^Dw^j(t)\Big(\dfrac{\partial \eta^j}{\partial x}(t)x_1(t)+\dfrac{\partial \eta^j}{\partial u}(t)\beta(t) \Big)\lambda_{j}(t) \Big\}\diffns t\Big]- E\Big[\dfrac{\partial h}{\partial x}(X(T),\alpha(T))A(T)x_1(T)\Big].\notag
\end{align}
By  \eqref{lambda1}, the It\^o formula, \eqref{derivassBSDE1} and \eqref{thneccond12}, we get
\begin{align}
&E\Big[\psi^\prime(Y(0))y_1(0)\Big] \notag\\
=&E\Big[A(0)y_1(0)\Big]\notag\\
=&E\Big[A(T)y_1(T)\Big] - E\Big[\int_0^T\Big\{A(t^-)\diff y_1(t)+y_1(t^-)\diff A(t) +\dfrac{\partial H}{\partial z}(t)z_1(t)\diff t\Big.\Big.\notag\\
&\Big.\Big.+\int_{\mathbb{R}_0}\nabla_kH(t,\zeta)k_1(t,\zeta)\nu_\alpha(\diffns  \zeta)\diff t+\sum_{j=1}^D\dfrac{\partial H}{\partial v^j}(t)  v^j_1(t)\lambda_{j}(t) \diff t  \Big\}\Big]\notag\\
=&E\Big[\dfrac{\partial h}{\partial x}(X(T),\alpha(T))A(T)x_1(T)+\int_0^T\Big\{A(t)\Big(\dfrac{\partial g}{\partial x}(t)x_1(t)  +\Big.\dfrac{\partial g}{\partial y}(t)y_1(t) +\dfrac{\partial g}{\partial z}(t)z_1(t) \Big.\Big.\Big. \notag\\
& +\int_{\mathbb{R}_0}\nabla_k g (t,\zeta)k_1(t,\zeta)\nu_\alpha(\diffns \zeta)   + \sum_{j=1}^D\dfrac{\partial g}{\partial v^j}(t)  v^j_1(t)\lambda_{j}(t)    +\dfrac{\partial g}{\partial u}(t)\beta(t)   \Big)-\dfrac{\partial H}{\partial y}(t)y_1(t) \notag\\
&   -\dfrac{\partial H}{\partial z}(t)z_1(t)-\int_{\mathbb{R}_0}\nabla_kH(t,\zeta)k_1(t,\zeta)\nu_\alpha(\diffns \zeta) -\sum_{j=1}^D\dfrac{\partial H}{\partial v^j}(t)  v^j_1(t)\lambda_{j}(t) \Big\}\diffns  t\Big] .\label{eqIprim3}
\end{align}
Substituting \eqref{eqIprim2} and \eqref{eqIprim3} into \eqref{eqIprim11}, we get
\begin{align}
&\dfrac{\diffns }{\diffns \ell}J^{(u+\ell \beta)}(t)\Big. \Big|_{\ell=0}\notag\\
=&E\Big[\int_0^T\Big( x_1(t)  \Big\{\dfrac{\partial f}{\partial x}(t) +A(t)\dfrac{\partial g}{\partial x}(t)+p(t)\dfrac{\partial b}{\partial x}(t)+q(t)\dfrac{\partial \sigma}{\partial x}(t)+\int_{\mathbb{R}_0} r(t,\zeta)\dfrac{\partial \gamma}{\partial x}(t,\zeta)\nu_\alpha(\diffns  \zeta)\notag\\
&+\sum_{j=1}^Dw^j(t)\dfrac{\partial \eta^j}{\partial x}(t)\lambda_{j}(t) -\dfrac{\partial H}{\partial x}(t)\Big\}+ y_1(t)  \Big\{\dfrac{\partial f}{\partial y}(t) +A(t)\dfrac{\partial g}{\partial y}(t)-\dfrac{\partial H}{\partial y}(t)\Big\}  \notag\\
&+z_1(t)  \Big\{\dfrac{\partial f}{\partial z}(t) +A(t)\dfrac{\partial g}{\partial z}(t)-\dfrac{\partial H}{\partial z}(t)\Big\} \notag
\end{align}
\begin{align}
&+\int_{\mathbb{R}_0}k_1(t,\zeta) \Big\{\nabla_kf(t,\zeta)+A(t)\nabla_kg(t,\zeta)-\nabla_kH(t,\zeta)\Big\}\nu_\alpha(\diffns \zeta) \notag\\
&+ \sum_{j=1}^Dv^j_1(t)  \Big\{\dfrac{\partial f}{\partial v^j}(t) +A(t)\dfrac{\partial g}{\partial v^j}(t)-\dfrac{\partial H}{\partial v^j}(t)\Big\} \notag\\
&+\beta(t)  \Big\{\dfrac{\partial f}{\partial u}(t) +A(t)\dfrac{\partial g}{\partial u}(t)+p(t)\dfrac{\partial b}{\partial u}(t)+q(t)\dfrac{\partial \sigma}{\partial u}(t)+\int_{\mathbb{R}_0} r(t,\zeta)\dfrac{\partial \gamma}{\partial u}(t,\zeta)\nu_\alpha(\diffns  \zeta)\notag\\
&+\sum_{j=1}^Dw^j(t)\dfrac{\partial \eta^j}{\partial u}(t)\lambda_{j}(t) \Big\}\Big)\diffns t\Big]. \label{eqIprim311}
\end{align}
By the definition of $H$, the coefficients of $x_1(t),y_1(t),z_1(t), k_1(t,\zeta)$ and $v_1(t)$ are all equal to zero in \eqref{eqIprim311}. Hence, if
$$\dfrac{\diffns }{\diffns \ell}J^{(u+\ell \beta)}(t)=0 \text{ for all bounded } \beta \in \mathcal{A}_{\mathcal{E}},$$
 it follows that
$$E\Big[\displaystyle \int_0^T\dfrac{\partial H}{\partial u}(t)\beta(t) \diff t \Big]=0 \text{  for all bounded } \beta \in \mathcal{A}_{\mathcal{E}}.$$
This holds in particular for $\beta \in \mathcal{A}_{\mathcal{E}}$ of the form $
\beta(t)=\beta_{t_0}(t,\omega)=\theta(\omega)\xi_{[t_0,T]}(t)$ for a fix $t_0\in [0,T)$, where $\theta(\omega)$ is a bounded $\mathcal{E}_{t_0}$-measurable random variable. Hence
$$
E\Big[\displaystyle \int_{t_0}^T\dfrac{\partial H}{\partial u}(t)\diff t\,\theta\Big]=0.
$$
Differentiating with respect to $t_0$, we have
$$
E\Big[\dfrac{\partial H}{\partial u}(s)\,\theta\Big]=0 \text{ for a.a., } t_0.
$$
Since the equality is true for all bounded $\mathcal{E}_{t_0}$-measurable random variable, we conclude that
$$
E\Big[\dfrac{\partial H}{\partial u}(t_0)|\mathcal{E}_{t_0}\Big]=0 \text{ for a.a., } t_0\in[0,T].
$$
This shows that (1) $\Rightarrow$ (2).

Conversely, using the fact that every bounded $\beta \in \mathcal{A}_{\mathcal{E}}$  can be approximated by a linear combinations of controls $\beta(t)$ of the form \eqref{eqbeta1}, the above argument can be reversed to show that  (2) $\Rightarrow$ (1).
\end{proof}

\section{A Malliavin calculus approach}\label{mallcalappro}
In this section, we shall give a method based on Malliavin calculus. This method was first introduced in \cite{MOZ12} when the state process is given by a SDE and extended in the stochastic partial differential equation (SPDE) case in \cite{MMPS13}. The set up is that of a Markov regime-switching forward-backward stochastic differential equations with jumps as in the previous sections and the notation are the same. For basic concepts of Malliavin calculus, we refere the reader to  \cite{DOP08, Nua06}.

In the following, let denote by $D^B_{t}F$ (respectively $D^{\widetilde{N}_\alpha}_{t,\zeta} F$ and $D^{\widetilde{\Phi}}_t F$ the Malliavin derivative in the direction of the Brownian motion $B$(respectively pure jump L\'evy process $\widetilde{N}_\alpha$ and the pure jump process $\widetilde{\Phi}$) of a given (Malliavin differentiable) random variable $F=F(\omega);\,\,\,\omega \in \Omega$. We denote by $\mathbb{D}_{1,2}$ the set of all random variables which are Malliavin differentiable with respect to $B(\cdot),\,\widetilde{N}_\alpha(\cdot,\cdot)$ and $\widetilde{\Phi}(\cdot)$. A crucial argument in the proof of our general maximum principle rests on duality formulas for the Malliavin derivatives $%
D_{t}$ and $D_{t,\zeta}$ (see for e.g., \cite{Nua06} and \cite{DOP08}):

\begin{align}
E\Big[ F\int_{0}^{T}\varphi (t)\diffns B(t)\Big] =&E\Big[
\int_{0}^{T}\varphi (t)D^B_{t}F\diffns  t\Big] , \label{dualformB1}\\
E\Big[ F\int_{0}^{T}\int_{\mathbb{R}_{0}}\psi (t,\zeta)\widetilde{N}_\alpha%
(\diffns t,\diffns \zeta)\Big] =&E\Big[ \int_{0}^{T}\int_{\mathbb{R}_{0}}\psi
(t,\zeta)D^{\widetilde{N}_\alpha}_{t,\zeta}F\nu_\alpha (\diffns \zeta)\diffns t\Big] ,\label{dualformN1}\\
E\Big[ F\int_{0}^{T}\varphi (t)\diffns \widetilde{\Phi}(t)%
\Big] =&E\Big[ \int_{0}^{T}\varphi
(t)D^{\widetilde{\Phi}}_{t}F\lambda\diffns t\Big],\label{dualformPhi1}
\end{align}
true for all Malliavin differentiable random variable $F$ and $\mathcal{F}_t$-predictable processes $\varphi$ and $\psi$ such that the the integrals on the right hand side converge absolutely.

We shall also need some basic properties of the Malliavin derivatives. Let $F \in \mathbb{D}_{1,2}$ be a $\mathcal{F}_s$-measurable random variable, then $
D^B_{t}F=D^{\widetilde{N}_\alpha}_{t,\zeta}F=D^{\widetilde{\Phi}}_{t}F=0 \text{ for all } t>s.$ We also have the following results known as the fundamental theorems of calculus
\begin{align}
D^B_s\Big(\int_0^t\varphi (s)\diff B(s) \Big)=&\varphi (s)1_{[0,t]}(s) +\int_s^tD_s\varphi (r)\diff B(r),\label{fundthm1}\\
D^{\widetilde{N}_\alpha}_{s,\zeta}\Big(\int_0^t \int_{ \mathbb{R}_0} \psi (s,\zeta)\widetilde{N}(\diffns s,\diffns \zeta)\Big)=&\psi (s,\zeta)1_{[0,t]}(s) +\int_s^t\int_{ \mathbb{R}_0} D^{\widetilde{N}}_{s,\zeta}  \psi (r,\zeta)\widetilde{N}_\alpha(\diffns r,\diffns \zeta)  ,\label{fundthm2}\\
D^{\widetilde{\Phi}}_{s}\Big(\int_0^t  \varphi (s)\diffns \widetilde{\Phi}( s)\Big)=&\varphi (s)1_{[0,t]}(s) +\int_s^t D^{\widetilde{\Phi}}_{s}  \varphi (r)\diffns \widetilde{\Phi}( r)  \label{fundthm3},
\end{align}
under the assumption that all the terms involved are well defined and belong to $\mathbb{D}_{1,2}$.

In view of the optimization problem \eqref{prob111}, we define the following processes: Suppose that for all $u\in $ $\mathcal{A}_{\mathcal{E}}$ the
processes
\begin{align}
\kappa(t):=& \nabla_xh (X (T),\alpha (T))\widetilde{A}(T)+ \nabla_x\varphi(X (T),\alpha (T)) \notag\\
&+
 \int_{t}^{T}\frac{\partial f }{\partial x }(s,X(s),\alpha(s),Y(s), Z(s), K(s,\cdot),V(s),u(s)) \diffns s, \label{eqKappa}
\end{align}
\begin{align}
 H_{0}\left(t,x,e_i,y,z,k,v,u,\widetilde{a},\kappa \right):= & \widetilde{a} g(t,x,e_i,y,z,k,v,u) + \kappa(t)b(t,x,e_i,u) +D_t^B\kappa(t)\sigma (t,x,e_i,u), \notag\\
+&\int_{\mathbb{R}_0}D_{t,\zeta}^{\widetilde{N}}\kappa(t)\gamma(t,x,e_i,u,\zeta)\nu_i(\diffns \zeta) +\sum_{j=1}^DD_{t}^{\widetilde{\Phi_j}}\kappa(t)\eta^j(t,x,e_i,u)\lambda_{ij} \label{eqH0}
\end{align}
\begin{align}
F(T):=&  \dfrac{\partial h}{\partial x} (X (T),\alpha (T))\tilde{A}(T)+ \dfrac{\partial \varphi}{\partial x} (X (T),\alpha (T))\label{eqcappsi}\\
\Theta(t,s):=& \frac{\partial H_{0}}{\partial x }(s)G(t,s)\label{eqcaptheta},
 \end{align}
\begin{align}
G(t,s):=&\exp\Big(\int_t^s\Big\{\frac{\partial b}{\partial x }(r)-\frac{1}{2}\Big(\frac{\partial \sigma}{\partial x }(r)\Big)^2  +\int_{\mathbb{R}_0}\Big(\ln \Big(1+\frac{\partial \gamma}{\partial x }\left(r,\zeta \right)\Big)-\frac{\partial \gamma}{\partial x }\left(r,\zeta\right)\Big)\nu_\alpha(\diffns \zeta) \notag\\
+& \sum_{j=1}^D \Big(\ln \Big(1+\frac{\partial \eta^j}{\partial x }(r)\Big)-\frac{\partial \eta^j}{\partial x }(r)\Big) \lambda_j(r) \Big\}\diffns r +\int_t^s\frac{\partial \sigma}{\partial x }\left(r\right)\diffns B(r)\notag\\
+&\int_t^s\int_{\mathbb{R}_0}\ln \Big(1+\frac{\partial \gamma}{\partial x }\left(r,\zeta\right)\Big)\widetilde{N}_\alpha(\diffns \zeta,\diffns r)+\sum_{j=1}^D\int_t^s \ln \Big(1+\frac{\partial \eta^j}{\partial x }(r)\Big)\diffns \widetilde{\Phi_j}(  r)
 \Big.\label{eqG}
\end{align}

are all well defined.  In \eqref{eqfirstp} and in the following we use the shorthand notation \\$H_0(t)=H_{0}\Big(t,X(t),\alpha(t),Y(t),Z(t),K(t,\cdot),V(t),u,\widetilde{A}(t),\kappa(t) \Big)$. We also assume that the following modified adjoint processes $(\tilde{p}(t),\tilde{q}(t),\tilde{r}(t,\zeta),\tilde{w}(t))$ and $\tilde{A}(t)$ given by
\begin{align}
\tilde{p}(t):= &\kappa(t)+\int_{t}^{T} \frac{\partial H_{0}}{\partial x }%
(s)G(t,s)\diffns s , \label{eqfirstp} \\
\tilde{q}(t):= &D^B_{t}\tilde{p}(t)   , \label{eqfirstq} \\
\tilde{r}(t,\zeta):= &D^{\widetilde{N}_\alpha}_{t,\zeta}\tilde{p}(t)  ,   \label{eqfirstr}\\
\tilde{w}^j(t):=&D_{t}^{\widetilde{\Phi_j}}\tilde{p}(t),\,\,\,j=1,\ldots,D  \label{eqfirstw}
\end{align}
and
\begin{equation}\left\{
\begin{array}{llll}
\diffns  \tilde{A} (t) & = & \dfrac{\partial H}{\partial y} (t) \diff t + \dfrac{\partial H}{\partial z} (t) \diffns B(t)+\displaystyle \int_{\mathbb{R}_0} \dfrac{\diffns \nabla_k H}{\diffns \nu(\zeta)} (t,\zeta) \,\widetilde{N}_\alpha(\diffns \zeta,\diffns t)\\
&& + \nabla_vH(t)\cdot\diffns \widetilde{\Phi}(t);\,\,\,t\in [0,T] \\
A(0) &=& \psi^\prime(Y(0)).\label{lambda111}
\end{array}\right.
\end{equation}
are well defined. Here the general Hamiltonian $H$ is given by \eqref{Hamiltoniansddg1} with $p,q,r,w$ replaced by $\tilde{p}, \tilde{q}, \tilde{r}, \tilde{w}$.
We can now state a general stochastic maximum principle for our control problem  \eqref{prob111}:

\begin{rem}\label{rem11111}
Assume that the coefficients of the control problem satisfy conditions for existence and uniqueness of the system \eqref{eqstateprocess1}-\eqref{eqassBSDE}, assume moreover that there are as in \textup{Remark \ref{rem111}}, then the processes given by \eqref{eqKappa}-\eqref{lambda111} are well defined.
\end{rem}

\begin{thm}\label{theomainneccon1G}
Let $u\in \mathcal{A}_{\mathcal{E}}$ with corresponding solutions $X(t)$ of \textup{\eqref{eqstateprocess1}}, $(Y(t),Z(t),K(t,\zeta),V(t))$ of \textup{\eqref{eqassBSDE}}, $\tilde{A}(t)$ of \textup{\eqref{lambda111}}, $\tilde{p}(t),\tilde{q}(t),\tilde{r}(t,\zeta),\tilde{w}^j(t)$ of \textup{\eqref{eqfirstp}-\eqref{eqfirstw}} and corresponding derivative processes $x_1(t)$ and $(y_1(t),z_1(t),k_1(t,\zeta),v_1(t))$ given by \textup{\eqref{derivstate1}} and \textup{\eqref{derivassBSDE1}} respectively. Suppose that
\textup{Assumptions \ref{assumces1}, \ref{assumces2}} and \textup{\ref{assumces3}} hold. Moreover, assume that the random variables $F(T),\Theta(t,s)$ given by \eqref{eqcappsi} and \eqref{eqcaptheta}, and $\dfrac{\partial f}{\partial x}(t)$ are Malliavin differentiable with respect to $B,\widetilde{N}$ and $\widetilde{\Phi}$. Furthermore, assume the following conditions

 \begin{align}
&E\Big[\int_0^T\Big\{ \Big(\dfrac{\partial \sigma}{\partial x}\Big)^2(t)x^2_1(t) +\Big(\dfrac{\partial \sigma}{\partial u}\Big)^2(t)\beta^2(t) +\int_{\mathbb{R}_0}\Big( \Big(\dfrac{\partial \gamma}{\partial x}\Big)^2(t,\zeta)x_1^2(t)+\Big(\dfrac{\partial \gamma}{\partial u}\Big)^2(t,\zeta)\beta^2(t)\Big)\nu_\alpha (\diffns \zeta)\notag\\
&\Big.+\sum_{j=1}^D \Big(\Big(\dfrac{\partial \eta^j}{\partial x}\Big)^2(t)x^2_1(t)+\Big(\dfrac{\partial \eta^j}{\partial u}\Big)^2(t)\beta^2(t)\Big)\lambda_j(t)\Big\} \diffns t\Big] <\infty, \label{thneccond114} \\
&E\Big[\int_0^T\int_0^T\Big\{ \Big(D^B_sF(T)\Big)^2+\int_{\mathbb{R}_0}\Big(D^{\widetilde{N}_\alpha}_{s,\zeta}F(T)\Big)^2 \nu_\alpha(\diffns \zeta)+\sum_{j=1}^D \Big(D^{\widetilde{\Phi}_j}_{s}F(T)\Big)^2\lambda_j(t)\Big\}\diffns s\diff t \Big] <\infty\notag,\\
&E\Big[\int_0^T\int_0^T\Big\{ \Big(D^B_s\Big( \frac{\partial f }{\partial x }(t)\Big)\Big)^2+\int_{\mathbb{R}_0}\Big(D^{\widetilde{N}_\alpha}_{s,\zeta}\Big( \frac{\partial f }{\partial x }(t)\Big)\Big)^2 \nu_\alpha(\diffns \zeta)+\sum_{j=1}^D \Big(D^{\widetilde{\Phi}_j}_{s}\Big( \frac{\partial f }{\partial x }(t)\Big)\Big)^2\lambda_j(t)\Big\}\diffns s\diff t \Big] <\infty ,\notag\\
&E\Big[\int_0^T\int_0^T\Big\{ \Big(D^B_s \Theta(t,s)\Big)^2+\int_{\mathbb{R}_0}\Big(D^{\widetilde{N}_\alpha}_{s,\zeta}\Theta(t,s)\Big)^2 \nu_\alpha(\diffns \zeta)+\sum_{j=1}^D \Big(D^{\widetilde{\Phi}_j}_{s}\Theta(t,s)\Big)^2\lambda_j(t)\Big\}\diffns s\diff t \Big] <\infty \notag.
\end{align}


Then the following are equivalent:

\textup{(1)} $\dfrac{\diffns}{\diffns \ell }J^{(u+\ell \beta)}(t)\Big. \Big|_{\ell=0}=0 \text{ for all bounded } \beta\in \mathcal{A}_{\mathcal{E}}.$

\textup{(2)} $E\Big[\dfrac{\partial H}{\partial u} (t,X(t),\alpha(t),Y(t), Z(t), K(t,\cdot),V(t),u,A(t),p(t),q(t),r(t,\cdot),w(t))_{u=u(t)}\Big. \Big| \mathcal{E}_t\Big] =0$
for a.a. $(t ,\omega)\in [0,T]\times \Omega$.

\end{thm}

Let us mention that if in addition of assumptions in Remark \ref{rem11111}, we suppose for example that the coefficient are twice continuously differentiable with the the second order derivative satisfying for example the assumptions in Remark \ref{rem11111} then $F(T),\Theta(t,s)$ and $\dfrac{\partial f}{\partial x}(t)$ are Malliavin differentiable with respect to $B,\widetilde{N}$ and $\widetilde{\Phi}$.
\begin{proof}
(1) $\Rightarrow$ (2). Assume that (1) holds then we have
\begin{align}
0=&\dfrac{\diffns }{\diffns \ell }J^{(u+\ell \beta)}(t)\Big. \Big|_{\ell=0}\notag\\
=&E\Big[\int_0^T \Big\{\dfrac{\partial f}{\partial x}(t)x_1(t)  +\dfrac{\partial f}{\partial y}(t)y_1(t)+\dfrac{\partial f}{\partial z}(t) z_1(t)+\int_{\mathbb{R}_0}\nabla_k f(t)k_1(t,\zeta)\nu_\alpha(\diffns \zeta)\Big.\notag \\
&\Big.+ \sum_{j=1}^D \dfrac{\partial f}{\partial v^j}(t)v_1^j(t)\lambda_j(t)+ \dfrac{\partial f}{\partial u}(t)\beta(t)\Big\} \diffns t+\dfrac{\partial \varphi}{\partial x}  (X(T),\alpha(T))x_1(T)+\psi^\prime(Y(0))y_1(0)\notag\\
&\Big.+\dfrac{\partial h}{\partial x} (X (T),\alpha (T))\Big(\tilde{A}(T)-\tilde{A}(T)\Big)x_1(T) \Big]. \label{eqIprim11g}
\end{align}

It follows from \eqref{derivstate1} and duality formula that for $F(T)$ defined by \eqref{eqcappsi} we get
\begin{align}
E\Big[F(T)x_1(T) \Big]=&E\Big[F(T) \Big\{\int_0^T\Big(\dfrac{\partial b}{\partial x}(t)x_1(t)+\dfrac{\partial b}{\partial u}(t)\beta(t)\Big)\diffns t+\int_0^T \Big(\dfrac{\partial \sigma}{\partial x}(t)x_1(t)+\dfrac{\partial \sigma}{\partial u}(t)\beta(t)\Big) \diffns B(t) \Big.\Big.\notag\\
& +\int_0^T \int_{\mathbb{R}_0} \Big(\dfrac{\partial \gamma}{\partial x}(t,\zeta)x_1(t) +\dfrac{\partial \gamma}{\partial u}(t,\zeta)\beta(t)\Big) \widetilde{N}_\alpha(\diffns \zeta,\diffns t) \notag\\
& +\sum_{j=1}^D\int_0^T \Big(\dfrac{\partial \eta^j}{\partial x}(t)x_1(t)-\dfrac{\partial \eta^j}{\partial u}(t)\beta(t) \Big)\diffns \widetilde{\Phi}_j(t) \Big\}\Big].\notag\\
=&E\Big[ \int_0^T \Big\{ F(T) \Big(\dfrac{\partial b}{\partial x}(t)x_1(t)+\dfrac{\partial b}{\partial u}(t)\beta(t)\Big)+ D_t^BF(T)\Big(\dfrac{\partial \sigma}{\partial x}(t)x_1(t)+\dfrac{\partial \sigma}{\partial u}(t)\beta(t)\Big) \Big.\notag\\
&+\int_{\mathbb{R}_0} D_{t,\zeta}^{\widetilde{N}_\alpha}F(T)\Big(\dfrac{\partial \gamma}{\partial x}(t,\zeta)x_1(t) +\dfrac{\partial \gamma}{\partial u}(t,\zeta)\beta(t)\Big) \nu_\alpha(\diffns \zeta)\notag\\
&\Big.+\sum_{j=1}^DD_t^{\widetilde{\Phi}_j}F(T)\Big(\dfrac{\partial \eta^j}{\partial x}(t)x_1(t)-\dfrac{\partial \eta^j}{\partial u}(t)\beta(t) \Big)\lambda_{j}(t) \Big\}\diffns t\Big]. \label{eqmalth1}
\end{align}
Similarly, we have
\begin{align}
E\Big[ \int_0^T \frac{\partial f }{\partial x }(t)x_1(t)\diffns t \Big]=&E\Big[\int_0^T \frac{\partial f }{\partial x }(t) \Big\{\int_0^t\Big(\dfrac{\partial b}{\partial x}(s)x_1(s)+\dfrac{\partial b}{\partial u}(s)\beta(s)\Big)\diffns s
 \Big.\Big.\notag\\
&+\int_0^t \Big(\dfrac{\partial \sigma}{\partial x}(s)x_1(s)+\dfrac{\partial \sigma}{\partial u}(s)\beta(s)\Big) \diffns B(s)\notag\\
 &+\int_0^t \int_{\mathbb{R}_0} \Big(\dfrac{\partial \gamma}{\partial x}(s,\zeta)x_1(s) +\dfrac{\partial \gamma}{\partial u}(s,\zeta)\beta(t)\Big) \widetilde{N}_\alpha(\diffns \zeta,\diffns s) \notag\\
 &+\sum_{j=1}^D\int_0^t \Big(\dfrac{\partial \eta^j}{\partial x}(s)x_1(s)-\dfrac{\partial \eta^j}{\partial u}(s)\beta(s) \Big)\diffns \widetilde{\Phi}_j(s) \Big\}\diffns t\Big].\notag
\end{align}
\begin{align}
=&E\Big[ \int_0^T\Big(\int_s^T \frac{\partial f }{\partial x }(t)\diffns t\Big) \Big(\dfrac{\partial b}{\partial x}(s)x_1(t)+\dfrac{\partial b}{\partial u}(s)\beta(s)\Big)\Big.\notag\\
&+ \Big(\int_s^T D_s^B\Big(\frac{\partial f }{\partial x }(t)\Big)\diffns t\Big)\Big(\dfrac{\partial \sigma}{\partial x}(s)x_1(s)+\dfrac{\partial \sigma}{\partial u}(s)\beta(s)\Big) \notag\\
&+\int_{\mathbb{R}_0} \Big(\int_s^T D_{s,\zeta}^{\widetilde{N}_\alpha}\Big(\frac{\partial f }{\partial x }(t)\Big)\diffns t\Big)\Big(\dfrac{\partial \gamma}{\partial x}(s,\zeta)x_1(s) +\dfrac{\partial \gamma}{\partial u}(s,\zeta)\beta(s)\Big) \nu_\alpha(\diffns \zeta)\notag\\
&\Big.+\sum_{j=1}^D \Big(\int_s^T D_s^{\widetilde{\Phi}_j}\Big(\frac{\partial f }{\partial x }(t)\Big)  \diffns t \Big) \Big(\dfrac{\partial \eta^j}{\partial x}(s)x_1(s)-\dfrac{\partial \eta^j}{\partial u}(s)\beta(s) \Big)\lambda_{j}(s) \Big\} \diffns s\Big].\notag
\end{align}
Changing the notation $s \leftrightarrow t$, this becomes
\begin{align}
=&E\Big[ \int_0^T\Big(\int_t^T \frac{\partial f }{\partial x }(s)\diffns s\Big) \Big(\dfrac{\partial b}{\partial x}(t)x_1(t)+\dfrac{\partial b}{\partial u}(t)\beta(t)\Big)\Big.\notag\\
&+ \Big(\int_t^T D_t^B\Big(\frac{\partial f }{\partial x }(s)\Big)\diffns s\Big)\Big(\dfrac{\partial \sigma}{\partial x}(t)x_1(t)+\dfrac{\partial \sigma}{\partial u}(t)\beta(t)\Big) \notag\\
&+\int_{\mathbb{R}_0} \Big(\int_t^T D_{t,\zeta}^{\widetilde{N}_\alpha}\Big(\frac{\partial f }{\partial x }(s)\Big)\diffns s\Big)\Big(\dfrac{\partial \gamma}{\partial x}(t,\zeta)x_1(t) +\dfrac{\partial \gamma}{\partial u}(t,\zeta)\beta(t)\Big) \nu_\alpha(\diffns \zeta)\notag\\
&\Big.+\sum_{j=1}^D \Big(\int_t^T D_t^{\widetilde{\Phi}_j}\Big(\frac{\partial f }{\partial x }(s)\Big)  \diffns s \Big) \Big(\dfrac{\partial \eta^j}{\partial x}(t)x_1(t)-\dfrac{\partial \eta^j}{\partial u}(t)\beta(t) \Big)\lambda_{j}(t) \Big\} \diffns t\Big]. \label{eqmalth2}
\end{align}
Combining \eqref{eqKappa}, \eqref{eqcappsi}, \eqref{eqmalth1} and \eqref{eqmalth2}, we have

\begin{align}
&E\Big[\int_0^T \Big(\dfrac{\partial f}{\partial x}(t)x_1(t)  + \dfrac{\partial f}{\partial u}(t)\beta(t)\Big) \diffns t+\dfrac{\partial \varphi}{\partial x}  (X(T),\alpha(T))x_1(T)\Big]\notag\\
=&E\Big[\int_0^T \dfrac{\partial f}{\partial x}(t)x_1(t)   \diffns t + F(T)x_1(T)+ \int_0^T \dfrac{\partial f}{\partial u}(t)\beta(t) \diffns t-\dfrac{\partial h}{\partial x}  (X(T),\alpha(T))\tilde{A}(T)x_1(T)\Big]\notag\\
=&E\Big[ \int_0^T \Big\{ \kappa(t) \Big(\dfrac{\partial b}{\partial x}(t)x_1(t)+\dfrac{\partial b}{\partial u}(t)\beta(t)\Big)+ D_t^B\kappa(t) \Big(\dfrac{\partial \sigma}{\partial x}(t)x_1(t)+\dfrac{\partial \sigma}{\partial u}(t)\beta(t)\Big) \Big.\notag\\
&+\int_{\mathbb{R}_0} D_{t,\zeta}^{\widetilde{N}_\alpha}\kappa(t) \Big(\dfrac{\partial \gamma}{\partial x}(t,\zeta)x_1(t) +\dfrac{\partial \gamma}{\partial u}(t,\zeta)\beta(t)\Big) \nu_\alpha(\diffns \zeta)\notag\\
&\Big.+\sum_{j=1}^DD_t^{\widetilde{\Phi}_j}\kappa(t) \Big(\dfrac{\partial \eta^j}{\partial x}(t)x_1(t)-\dfrac{\partial \eta^j}{\partial u}(t)\beta(t) \Big)\lambda_{j}(t) \Big\}\diffns t\notag\\
&+ \int_0^T \dfrac{\partial f}{\partial u}(t)\beta(t) \diffns t-\dfrac{\partial h}{\partial x}  (X(T),\alpha(T))\tilde{A}(T)x_1(T) \Big].\label{eqmalth3}
\end{align}
By the It\^o formula and \eqref{lambda111}, we have similarly to \eqref{eqIprim3}
\begin{align*}
E\Big[\psi^\prime(Y(0))y_1(0)\Big] =&E\Big[\tilde{A}(0)y_1(0)\Big]\notag\\
=&E\Big[\dfrac{\partial h}{\partial x}(X(T),\alpha(T))\tilde{A}(T)x_1(T)\Big] + E\Big[\int_0^T\Big\{\tilde{A}(t)\Big(\dfrac{\partial g}{\partial x}(t)x_1(t)  +\Big.\dfrac{\partial g}{\partial y}(t)y_1(t) \Big.\Big.\Big. \notag\\
&  +\dfrac{\partial g}{\partial z}(t)z_1(t)+\int_{\mathbb{R}_0}\nabla_k g (t)k_1(t,\zeta)\nu_\alpha(\diffns \zeta)   + \sum_{j=1}^D\dfrac{\partial g}{\partial v^j}(t)  v^j_1(t)\lambda_{j}(t)    \notag\\
&   +\dfrac{\partial g}{\partial u}(t)\beta(t)   \Big)-\dfrac{\partial H}{\partial y}(t)y_1(t) -\dfrac{\partial H}{\partial z}(t)z_1(t)-\int_{\mathbb{R}_0}\nabla_kH(t)k_1(t,\zeta)\nu_\alpha(\diffns \zeta) \notag\\
& -\sum_{j=1}^D\dfrac{\partial H}{\partial v^j}(t)  v^j_1(t)\lambda_{j}(t) \Big\}\diffns  t\Big] .
\end{align*}
But
\begin{align*}
\dfrac{\partial H}{\partial y}(t)=\dfrac{\partial f}{\partial y}(t)+\tilde{A}(t)\dfrac{\partial g}{\partial y}(t);&\,\,\, \dfrac{\partial H}{\partial z}(t)=\dfrac{\partial f}{\partial z}(t)+\tilde{A}(t)\dfrac{\partial g}{\partial z}(t)\\
\nabla_k H(t)=\nabla_kf(t)+\tilde{A}(t)\nabla_kg(t);& \,\,\,\dfrac{\partial H}{\partial v^j}(t)=\dfrac{\partial f}{\partial v^j}(t)+\tilde{A}(t)\dfrac{\partial g}{\partial v^j}(t),\,\,j=1,\ldots,D.
\end{align*}
Hence we have
\begin{align}
E\Big[\psi^\prime(Y(0))y_1(0)\Big] =&E\Big[\dfrac{\partial h}{\partial x}(X(T),\alpha(T))\tilde{A}(T)x_1(T)\Big] + E\Big[\int_0^T\Big\{\tilde{A}(t)\Big(\dfrac{\partial g}{\partial x}(t)x_1(t)+\dfrac{\partial g}{\partial u}(t)\beta(t)   \Big)\diffns t\notag\\
&-\int_0^T\Big\{ \dfrac{\partial f}{\partial y}(t)y_1(t) +\dfrac{\partial f}{\partial z}(t)z_1(t)   +\int_{\mathbb{R}_0}\nabla_k f (t,)k_1(t,\zeta)\nu_\alpha(\diffns \zeta)  \notag\\
& + \sum_{j=1}^D\dfrac{\partial g}{\partial v^j}(t)  v^j_1(t)\lambda_{j}(t)  \Big\}\diffns t\Big].
\label{eqmalth4}
\end{align}
Substitution \eqref{eqmalth1}-\eqref{eqmalth4} into \eqref{eqIprim11g}, we get
\begin{align}
0=&\dfrac{\diffns }{\diffns \ell}J^{(u+\ell \beta)}(t)\Big. \Big|_{\ell=0}\notag\\
=&E\Big[ \int_0^T \Big\{ \kappa(t) \dfrac{\partial b}{\partial x}(t)+ D_t^B\kappa(t)\dfrac{\partial \sigma}{\partial x}(t) +\int_{\mathbb{R}_0} D_{t,\zeta}^{\widetilde{N}_\alpha}\kappa(t)  \dfrac{\partial \gamma}{\partial x}(t,\zeta)  \nu_\alpha(\diffns \zeta)\notag\\
&+\sum_{j=1}^DD_t^{\widetilde{\Phi}_j}\kappa(t) \dfrac{\partial \eta^j}{\partial x}(t)+\tilde{A}(t)\dfrac{\partial g}{\partial x}(t)\Big\}x_1(t)\diffns t\Big]\notag\\
&+E\Big[ \int_0^T \Big\{ \kappa(t) \dfrac{\partial b}{\partial u}(t)+ D_t^B\kappa(t)\dfrac{\partial \sigma}{\partial u}(t) +\int_{\mathbb{R}_0} D_{t,\zeta}^{\widetilde{N}_\alpha}\kappa(t)  \dfrac{\partial \gamma}{\partial u}(t,\zeta)  \nu_\alpha(\diffns \zeta)\notag
\end{align}
\begin{align}
&+\sum_{j=1}^DD_t^{\widetilde{\Phi}_j}\kappa(t) \dfrac{\partial \eta^j}{\partial u}(t)+\dfrac{\partial f}{\partial u}(t)+\tilde{A}(t)\dfrac{\partial g}{\partial u}(t)\Big\}\beta(t)\diffns t\Big]. \label{eqmalth5}
\end{align}
Equation \eqref{eqmalth5} holds for all $\beta \in \mathcal{A}_{\mathcal{E}}$. In particular, if we apply this to $
\beta_{\theta}=\beta_{\theta}(s)=\theta(\omega)\chi_{(t,t+h]}(s),$ where $\theta(\omega)$ is $\mathcal{E}_t$-measure and $0\leq t\leq t+h \leq T.$ Hence we get by \eqref{derivstate1} that $
x_1=x_1^{(\beta_{\theta})}(s)=0 \text{ for } 0\leq s\leq t.$ Therefore \eqref{eqmalth5} can be rewritten as
\begin{align}
J_1(h)+J_2(h)=0,\label{eqj1j2}
\end{align}
where
\begin{align}
J_1(h)=&E\Big[ \int_t^T \Big\{ \kappa(s) \dfrac{\partial b}{\partial x}(s)+ D_s^B\kappa(s)\dfrac{\partial \sigma}{\partial x}(s) +\int_{\mathbb{R}_0} D_{s,\zeta}^{\widetilde{N}_\alpha}\kappa(s)  \dfrac{\partial \gamma}{\partial x}(s,\zeta)  \nu_\alpha(\diffns \zeta)\notag\\
&+\sum_{j=1}^DD_t^{\widetilde{\Phi}_j}\kappa(s) \dfrac{\partial \eta^j}{\partial x}(s)+\tilde{A}(s)\dfrac{\partial g}{\partial x}(s)\Big\}x_1(s)\diffns s\Big],\label{eqmalth51}\\
J_2(h)=&E\Big[\theta \int_t^{t+h} \Big\{ \kappa(s) \dfrac{\partial b}{\partial u}(s)+ D_t^B\kappa(s)\dfrac{\partial \sigma}{\partial u}(s) +\int_{\mathbb{R}_0} D_{s,\zeta}^{\widetilde{N}_\alpha}\kappa(s)  \dfrac{\partial \gamma}{\partial u}(s,\zeta)  \nu_\alpha(\diffns \zeta)\notag\\
&+\sum_{j=1}^DD_s^{\widetilde{\Phi}_j}\kappa(s) \dfrac{\partial \eta^j}{\partial u}(s)+\dfrac{\partial f}{\partial u}(s)+\tilde{A}(s)\dfrac{\partial g}{\partial u}(s)\Big\}\diffns s\Big] .\label{eqmalth52}
\end{align}
Note that for $x_1(s)=x_1^{(\beta_{\theta})}(s)$ we have, if $s\geq t+h$
\begin{equation*}
\diffns x_1(t) =x_1(t-) \Big\{\dfrac{\partial b}{\partial x}(t)\diffns t+\dfrac{\partial \sigma}{\partial x}(t)\diffns B(t)+\int_{\mathbb{R}_0} \dfrac{\partial \gamma}{\partial x}(t,\zeta)\widetilde{N}_\alpha(\diffns t,\diffns \zeta) + \dfrac{\partial \eta}{\partial x}(t)\cdot\diffns \widetilde{\Phi}(t) \Big\};\,\,t\in[0,T].
\end{equation*}
Hence by the It\^o formula we have $x_1(s)=x_1(t+h)G(t+h,s); \,\,\,s\geq t+h,$
where $G$ is defined by \eqref{eqG}. Since $G$ does not depend on $h$, it follows by the definition of $H_0$ (see \eqref{eqH0}) that
\begin{align*}
J_1(h)= E\Big[ \int_t^T \dfrac{\partial H_0}{\partial x}(s)x_1(s)\diffns s\Big]=E\Big[ \int_t^{t+h} \dfrac{\partial H_0}{\partial x}(s)x_1(s)\diffns s\Big]+E\Big[ \int_{t+h}^T \dfrac{\partial H_0}{\partial x}(s)x_1(s)\diffns s\Big].
\end{align*}
Differentiating with respect to $h$ at $h=0$ gives
\begin{align}
\frac{\diffns }{\diffns h} J_1(h)\Big|_{h=0} =\frac{\diffns }{\diffns h}E\Big[ \int_t^{t+h} \dfrac{\partial H_0}{\partial x}(s)x_1(s)\diffns s\Big]_{h=0}+\frac{\diffns }{\diffns h}E\Big[ \int_{t+h}^T \dfrac{\partial H_0}{\partial x}(s)x_1(s)\diffns s\Big]_{h=0}.\label{eqmalth512}
\end{align}
Since $x_1(t)=0$, we get $
\dfrac{\diffns }{\diffns h}E\Big[ \displaystyle \int_t^{t+h} \dfrac{\partial H_0}{\partial x}(s)x_1(s)\diffns s\Big]_{h=0}=0.$
 Using the definition of $x_1(s)$, we have
 \begin{align}
\frac{\diffns }{\diffns h}E\Big[ \int_{t+h}^T \dfrac{\partial H_0}{\partial x}(s)x_1(s)\diffns s\Big]_{h=0}=&\frac{\diffns }{\diffns h}E\Big[ \int_{t+h}^T \dfrac{\partial H_0}{\partial x}(s)x_1(t+h)G(t+h,s) \diffns s\Big]_{h=0}\notag\\
=&\int_{t}^T\frac{\diffns }{\diffns h}E\Big[  \dfrac{\partial H_0}{\partial x}(s)x_1(t+h)G(t+h,s) \Big]_{h=0}\diffns s\notag\\
=&\int_{t}^T\frac{\diffns }{\diffns h}E\Big[  \dfrac{\partial H_0}{\partial x}(s)x_1(t+h)G(t,s) \Big]_{h=0}\diffns s, \label{eqmalth513}
\end{align}
where $X_1(t+h)$ is given by

 \begin{align}
x_1(t+h) =&\int_t^{t+h}\Big( x_1(r-) \Big\{\dfrac{\partial b}{\partial x}(r)\diffns r+\dfrac{\partial \sigma}{\partial x}(r) \diffns B(r)+\int_{\mathbb{R}_0} \dfrac{\partial \gamma}{\partial x}(r,\zeta)\widetilde{N}_\alpha(\diffns t,\diffns \zeta) + \dfrac{\partial \eta}{\partial x}(r)\cdot\diffns \widetilde{\Phi}(r) \Big\}\notag\\
+&\theta \Big\{\dfrac{\partial b}{\partial u}(r)\diffns r+\dfrac{\partial \sigma}{\partial u}(r) \diffns B(r)+\int_{\mathbb{R}_0} \dfrac{\partial \gamma}{\partial u}(r,\zeta)\widetilde{N}_\alpha(\diffns t,\diffns \zeta) + \dfrac{\partial \eta}{\partial u}(r)\cdot\diffns \widetilde{\Phi}(r) \Big\}\Big).\label{eqmalth514}
\end{align}
Therefore, by \eqref{eqmalth513} and \eqref{eqmalth514} $
\frac{\diffns }{\diffns h} J_1(h)\Big|_{h=0}=J_{1,1}(0)+J_{1,2}(0),$ with
 \begin{align}
J_{1,1}(0)=&\int_{t}^T\frac{\diffns }{\diffns h}E\Big[  \dfrac{\partial H_0}{\partial x}(s) G(t,s) \theta \int_t^{t+h}\Big\{\dfrac{\partial b}{\partial u}(r)\diffns r+\dfrac{\partial \sigma}{\partial u}(r) \diffns B(r)\notag\\
&+\int_{\mathbb{R}_0} \dfrac{\partial \gamma}{\partial u}(r,\zeta)\widetilde{N}_\alpha(\diffns t,\diffns \zeta) + \dfrac{\partial \eta}{\partial u}(r)\cdot\diffns \widetilde{\Phi}(r) \Big\} \Big]_{h=0}\diffns s
\label{eqmalth514}\\
J_{1,2}(0)=&\int_{t}^T\frac{\diffns }{\diffns h}E\Big[  \dfrac{\partial H_0}{\partial x}(s) G(t,s) \int_t^{t+h}x_1(r-)\Big\{\dfrac{\partial b}{\partial x}(r)\diffns r+\dfrac{\partial \sigma}{\partial x}(r) \diffns B(r)\notag\\
&+\int_{\mathbb{R}_0} \dfrac{\partial \gamma}{\partial x}(r,\zeta)\widetilde{N}_\alpha(\diffns t,\diffns \zeta) + \dfrac{\partial \eta}{\partial x}(r)\cdot\diffns \widetilde{\Phi}(r) \Big\} \Big]_{h=0}\diffns s.
\label{eqmalth515}
\end{align}
Since $x_1(t)=0$, we have $J_{1,2}(0)=0.$ We conclude that $\frac{\diffns }{\diffns h} J_1(h)\Big|_{h=0}=J_{1,1}(0).$ Using once more the duality formula, we get from \eqref{eqcaptheta} that
 \begin{align}
J_{1,1}(0)=&\int_{t}^T\frac{\diffns }{\diffns h}E\Big[  \theta \int_t^{t+h}\Big\{\dfrac{\partial b}{\partial u}(r) \Theta(t,s)+\dfrac{\partial \sigma}{\partial u}(r) D_r^B\Theta(t,s)\notag\\
&+\int_{\mathbb{R}_0} \dfrac{\partial \gamma}{\partial u}(r,\zeta)D_{r,\zeta}^{\widetilde{N}_\alpha} \Theta(t,s) \nu_\alpha(\diffns \zeta) + \sum_{j=1}^D \dfrac{\partial \eta^j}{\partial u}(r) D_{r}^{\widetilde{\Phi}_j} \Theta(t,s) \Big\} \diffns r \Big]_{h=0}\diffns s \notag\\
=&\int_{t}^TE\Big[  \Big\{\dfrac{\partial b}{\partial u}(t) \Theta(t,s)+\dfrac{\partial \sigma}{\partial u}(t) D_t^B\Theta(t,s)\notag\\
&+\int_{\mathbb{R}_0} \dfrac{\partial \gamma}{\partial u}(t,\zeta)D_{t,\zeta}^{\widetilde{N}_\alpha} \Theta(t,s) \nu_\alpha(\diffns \zeta) + \sum_{j=1}^D \dfrac{\partial \eta^j}{\partial u}(t) D_{t}^{\widetilde{\Phi}_j} \Theta(t,s)\lambda_j(t) \Big\} \Big]\diffns s\label{eqmalth516}
\end{align}
On the other hand, differentiating \eqref{eqmalth52} with respect to $h$ at $h=0$, we have
\begin{align}
\frac{\diffns }{\diffns h} J_2(h)\Big|_{h=0}=&E\Big[ \theta\Big\{ \kappa(t) \dfrac{\partial b}{\partial u}(t)+ D_t^B\kappa(t)\dfrac{\partial \sigma}{\partial u}(t) +\int_{\mathbb{R}_0} D_{t,\zeta}^{\widetilde{N}_\alpha}\kappa(t)  \dfrac{\partial \gamma}{\partial u}(t,\zeta)  \nu_\alpha(\diffns \zeta)\notag\\
&+\sum_{j=1}^DD_t^{\widetilde{\Phi}_j}\kappa(t) \dfrac{\partial \eta^j}{\partial u}(t)\lambda_j(t)+\dfrac{\partial f}{\partial u}(t)+\tilde{A}(t)\dfrac{\partial g}{\partial u}(t)\Big\}\Big] .\label{eqmalth5211}
\end{align}
Moreover, differentiating \eqref{eqj1j2} with respect to $h$ at $h=0$ gives

\begin{align}
&E\Big[ \theta\Big\{ \Big( \kappa(t) + \int_{t}^T  \Theta(t,s)\diffns s  \Big) \dfrac{\partial b}{\partial u}(t)+ D_t^B\Big( \kappa(t) + \int_{t}^T  \Theta(t,s)\diffns s  \Big) \dfrac{\partial \sigma}{\partial u}(t)\notag
\end{align}
\begin{align}
& +\int_{\mathbb{R}_0} D_{t,\zeta}^{\widetilde{N}_\alpha}\Big( \kappa(t) + \int_{t}^T  \Theta(t,s)\diffns s  \Big)  \dfrac{\partial \gamma}{\partial u}(t,\zeta)  \nu_\alpha(\diffns \zeta)\notag\\
&+\sum_{j=1}^DD_t^{\widetilde{\Phi}_j}\Big( \kappa(t) + \int_{t}^T  \Theta(t,s)\diffns s  \Big) \dfrac{\partial \eta^j}{\partial u}(t)\lambda_j(t)+\dfrac{\partial f}{\partial u}(t)+\tilde{A}(t)\dfrac{\partial g}{\partial u}(t)\Big\}\Big]=0 .\label{eqmalth5212}
\end{align}
Using \eqref{eqfirstq}-\eqref{eqfirstw} and \eqref{Hamiltoniansddg1} with $p,q,r,w$ replaced by $\tilde{p}, \tilde{q}, \tilde{r}, \tilde{w}$, we get

\begin{align*}
E\Big[ \theta \dfrac{\partial H}{\partial u}\Big(t,X(t),\alpha(t),Y(t), Z(t), K(t,\cdot),V(t),u,A(t),p(t),q(t),r(t,\cdot),w(t)\Big)_{u=u(t)}\Big]=0. 
\end{align*}
Since this holds for all $\mathcal{E}_t$-measurable random variables $\theta$, we conclude that
{\small
\begin{align}
&E\Big[\dfrac{\partial H}{\partial u} (t,X(t),\alpha(t),Y(t), Z(t), K(t,\cdot),V(t),u,A(t),p(t),q(t),r(t,\cdot),w(t))_{u=u(t)} \Big. \Big| \mathcal{E}_t\Big] =0.\label{eqmalth5214}
\end{align}
}
(2) $\Rightarrow$ (1). Conversely, assume that there exist $u\in\mathcal{A}_{\mathcal{E}}$ such that \eqref{eqmalth5214} holds. Then by reversing the previous argument, we obtain that (1) holds for $\beta_{\theta}(s)=\theta(\omega)\chi_{(t,t+h]}(s) \in \mathcal{A}_{\mathcal{E}} $, where $\theta$ is bounded and  $\mathcal{E}_t$-measurable. Then \eqref{eqj1j2} holds for all linear combinations of $\beta_{\theta}$. Since all bounded $\beta \in \mathcal{A}_{\mathcal{E}} $ can be approximated pointwise boundedly in $(t,\omega)$ by such linear combination, it follows that \eqref{eqj1j2} is satisfied for all bounded $\beta \in \mathcal{A}_{\mathcal{E}} $. Thus reversing the remaining part of the previous proof, we get $\dfrac{\diffns }{\diffns \ell}J^{(u+\ell \beta)}(t)\Big. \Big|_{\ell=0}=0$ for all bounded $\beta \in \mathcal{A}_{\mathcal{E}}$.
\end{proof}

\section{Applications}\label{appli1}
\begin{appli} we shall  apply the results obtained to study an optimal control problem for Markov regime-switching with non-concave value function. Suppose that the state process $X(t)=X^{(u)}(t,\omega);\,\,0 \leq t \leq T,\,\omega \in \Omega$ is a controlled Markov regime-switching jump-diffusion of the form
\begin{equation}
\diffns X (t) = u(t)\Big\{ \sigma(t) \diff B(t)+ \displaystyle \int_{\mathbb{R}_0}\gamma(t,\zeta)\,\widetilde{N}(\diffns \zeta,\diffns t)\Big\},\,\,\,\, t \in [ 0,T], \,\,\,\,X(0)=  0 \label{eqstateprocessappl1}
\end{equation}
where $T>0$ is a given constant. $u(\cdot)$ is the control process. We shall assume here that $\widetilde{N}_\alpha=\widetilde{N}$ for any state of the Markov chain. Let us introduce the performance functional
{\small
\begin{align}
J(u)=E\Big[   \int_0^T\Big\{C_1(\alpha(t))u(t)+C_2(\alpha(t))u^2(t)+C_3(\alpha(t))X^2(t)\Big\} \diffns t +C_4(\alpha(T))X^2(T)\Big] .\label{perfuncappl1}
\end{align}
}
In this case, we have that
$$
 f(t,x,\alpha,y,z,k,v,u)=C_1(\alpha)u+C_2(\alpha)u^2+C_3(\alpha)x^2, \,\,\,\varphi(x,\alpha)=C_4(\alpha)x^2,\,\,\,\  g= \psi=0$$
\begin{align*}
 &\kappa(t)=2 C_4(\alpha(T))X(T)+2\int_t^TC_3(\alpha(s))X(s)\diffns s,\,\,\,\, A(t)=G(t,s)=0,\\
&H_0\left(t,x,e_i,y,z,k,v,u,\widetilde{a},\kappa\right)=D_t^B\kappa(t) u\sigma(t)+\int_{\mathbb{R}_0} D_{t,\zeta}^{\widetilde{N}_\alpha}\kappa(t)\gamma(t,\zeta)u\nu_i(\diffns \zeta),
\end{align*}
\begin{align*}
H\left(t,x,e_i,y,z,k,v,u,a,p,q,r,w\right)= &C_1(e_i)u+C_2(e_i)u^2+C_3(e_i)x^2 +  \tilde{q}(t) \sigma (t)u\\
&+\int_{\mathbb{R}_0}\tilde{r}(t,\zeta)\gamma(t,\zeta)u\nu_i(\diffns \zeta),
\end{align*}
with the modified adjoint processes given by
\begin{align*}
\tilde{p}(t)=& \kappa(t)+\int_t^T\frac{\partial H_0}{\partial x}(s)G(t,s) \diffns s=\kappa(t),\,\,\,\tilde{q}(t)= D_t^B\kappa(t),\\
\tilde{r}(t,\zeta)=&D_{t,\zeta}^{\widetilde{N}_\alpha}\kappa(t),\,\,\,   \tilde{w}^j(t)=D_{t}^{\widetilde{\Phi_j}}\kappa(t),\,\,\,j=1,\ldots,D.
\end{align*}
\begin{rem}
The Hamiltonian in this case is not concave and therefore Theorem \ref{mainressuf1} cannot be applied. However, using the Malliavin calculus approach we are able to derive a stochastic maximum principle.
\end{rem}

\begin{thm}\label{thmappli1}
Assume that the state process is given by \textup{\eqref{eqstateprocessappl1}} and let the performance functional be given by \textup{\eqref{perfuncappl1}}. Moreover, assume that $\alpha(t)$ is a two-state Markov chain and that $\mathcal{E}_t=\mathcal{F}_t  \text{ for all } t\in [0,T]$. Assume that an optimal control exists. Then $u^\ast$ is an optimal control for \textup{ \eqref{prob111}} iff
\begin{align}\label{equopappli1}
u^\ast(t)=&\dfrac{-C_1(1)}{2C_2(1)+ 2\Gamma(t,T,1)\Big(\sigma^2(t)+\int_{\mathbb{R}_0}\gamma^2(t,\zeta) \nu(\diffns \zeta)\Big)}\chi_{\{\alpha(t-)=1\}}\notag\\
&+\dfrac{-C_1(2)}{2C_2(2)+ 2\Gamma(t,T,2)\Big(\sigma^2(t)+\int_{\mathbb{R}_0}\gamma^2(t,\zeta) \nu(\diffns \zeta)\Big)}\chi_{\{\alpha(t-)=2\}},
\end{align}
where
\begin{align}\label{eqGamopappli1}
\Gamma(t,T,1)=&C_4(1)+C_3(1)(T-t)+C_3(2,1)\frac{\lambda_{1,2}}{\lambda_{1,2} +\lambda_{2,1}}(T-t)\notag\\
&+\frac{\lambda_{1,2}\Big\{C_4(2,1)(\lambda_{1,2} +\lambda_{2,1})-C_3(2,1)\Big\}}{(\lambda_{1,2} +\lambda_{2,1})^2}\Big\{1-e^{(\lambda_{1,2} +\lambda_{2,1})(t-T)}\Big\}
\end{align}
and $\Gamma(t,T,2)$ is computed in a similar way.
\end{thm}

\begin{proof}

The condition (2) in Theorem \ref{theomainneccon1G} for an optimal control $\hat{u}(t)$ is one of the two
\begin{align}
E\Big[C_1(\alpha(t))+2C_2(\alpha(t))u(t)+\sigma(t)\tilde{q}(t)+\int_{\mathbb{R}_0}\tilde{r}(t,\zeta)\gamma(t,\zeta)\nu_\alpha(\diffns \zeta)\Big|\mathcal{E}_t\Big]=0,\label{hamilappl11}\\
E\Big[C_1(\alpha(t))+2C_2(\alpha(t))u(t)+\sigma(t)D_t^B\tilde{p}(t)+\int_{\mathbb{R}_0}D_{t,\zeta}^{\widetilde{N}_\alpha}\tilde{p}(t)\gamma(t,\zeta)\nu_\alpha(\diffns \zeta)\Big|\mathcal{E}_t\Big]=0. \label{hamilappl12}
\end{align}
Equation \eqref{hamilappl12} can be seen as a partial information, Markov switching Malliavin-differential type
equation in the unknown random variable $\tilde{p}(t)$. A similar equation was solved in \cite{OS091} in a non regime switching case when $\mathcal{E}_t=\mathcal{F}_t$. From now on, we set $\mathcal{E}_t=\mathcal{F}_t  \text{ for all } t\in [0,T]$ and that $\alpha$ is a two-state Markov chain. Using the fundamental theorem of calculus, we have

\begin{align*}
\tilde{q}(t)=D_t^B\tilde{p}(t)=&2C_4(\alpha(T))D_t^BX(T)+2\int_t^TC_3(\alpha(s))D_t^BX(s)\diffns s\\
=&2C_4(\alpha(T))\Big\{\int_t^T D_t^B \Big(u(r) \sigma(r)\Big)\diffns B(r) +u(t)\sigma(t)\\
& +\int_t^T\int_{\mathbb{R}_0}D_t^B\Big(u(r)\gamma(r,\zeta)\Big)\widetilde{N}_\alpha(\diffns \zeta,\diffns r)\Big\}\notag\\
&+ 2\int_t^T C_3(\alpha(s))\Big\{ \int_t^s D_t^B \Big(u(r)\sigma(r)\Big)\diffns B(r) +u(t)\sigma(t) \notag\\
&+ \int_t^s\int_{\mathbb{R}_0}D_t^B\Big(u(r)\gamma(r,\zeta)\Big)\widetilde{N}_\alpha(\diffns \zeta,\diffns r)\Big\}\diffns s.
\end{align*}
Using integration by parts formula (or product rule) we get
\begin{align}
\tilde{q}(t)=D_t^B\tilde{p}(t)=&2\Big\{C_4(\alpha(t))u(t)\sigma(t)+\int_t^TC_4(\alpha(r)) D_t^B \Big(u(r)\sigma(r)\Big)\diffns B(r) \notag\\
&+\int_t^T\int_{\mathbb{R}_0}C_4(\alpha(r))D_t^B\Big(u(r)\gamma(r,\zeta)\Big)\widetilde{N}_\alpha(\diffns \zeta,\diffns r)\notag\\
&+\int_t^TD_t^BX(r)\sum_{j=1,i\neq j}^D \lambda_{i,j}(C_4(j)-C_4(i))\chi_{(\alpha(r)=i)}\diffns r\notag\\
&+\int_t^TD_t^BX(r)\sum_{j=1,i\neq j}^D \lambda_{i,j}(C_4(j)-C_4(i))\chi_{(\alpha(r)=i)}\diffns m_{ij}(t)\Big\}\notag \\
&+2\Big\{\int_t^T \Big( C_3(\alpha(t))u(t)\sigma(t) +\int_t^sC_3(\alpha(r)) D_t^B\Big(u(r) \sigma(r)\Big)\diffns B(r) \notag \\
&+\int_{\mathbb{R}_0}\int_t^sC_3(\alpha(r))D_t^B\Big(u(r)\gamma(r,\zeta)\Big)\widetilde{N}_\alpha(\diffns \zeta,\diffns r)\notag\\
&+\int_t^sD_t^BX(r)\sum_{j=1,i\neq j}^D \lambda_{i,j}(C_3(j)-C_3(i))\chi_{(\alpha(r)=i)}\diffns r  \notag \\
&+\int_t^sD_t^BX(r)\sum_{j=1,i\neq j}^D \lambda_{i,j}(C_3(j)-C_3(i))\chi_{(\alpha(r)=i)}\diffns m_{ij}(t)\Big)\diffns s\Big\}.\label{eqqtilde1}
\end{align}

Taking conditional expectation with respect to $\mathcal{F}_t$, we have
\begin{align}\label{eqqtildecon1}
E\Big[\tilde{q}(t)\Big|\mathcal{F}_t\Big]=&2C_4(\alpha(t))u(t)\sigma(t)+2\int_t^Tu(t)\sigma(t)\sum_{j=1,i\neq j}^D \lambda_{i,j}(C_4(j)-C_4(i))E\Big[\chi_{(\alpha(r)=i)}\Big|\mathcal{F}_t\Big]\diffns r\notag\\
&+2C_3(\alpha(t))u(t)\sigma(t)(T-t) \notag\\
&+2\int_t^T\int_t^su(t)\sigma(t) \sum_{j=1,i\neq j}^D \lambda_{i,j}(C_3(j)-C_3(i))E\Big[\chi_{(\alpha(r)=i)}\Big|\mathcal{F}_t\Big]\diffns r \diff s.
\end{align}
Let $\alpha(t)=e_1$ and for $n=1,2,3,4$, let $C_n(i)$ be the value of the function $C_n$ at $1$. Define $C_n(2,1)$ for $n=1,2,3,4$ by $C_n(2,1):=C_n(2)-C_n(1).$  We have

\begin{align}
E\Big[\tilde{q}(t)\Big|\mathcal{F}_t\Big]=&2C_4(1)u(t)\sigma(t)+2\int_t^Tu(t)\sigma(t) \Big(\lambda_{1,2}(C_4(2)-C_4(1))E\Big[\chi_{(\alpha(r)=1)}\Big|\alpha(t)=1\Big]\notag\\
& +\lambda_{2,1}(C_4(1)-C_4(2))E\Big[\chi_{(\alpha(r)=2)}\Big|\alpha(t)=1\Big]\Big)\diffns r+2 C_3(1)u(t)\sigma(t)(T-t) \notag\\
&+2\int_t^T\int_t^su(t)\sigma(t)\Big(\lambda_{1,2}(C_3(2)-C_3(1))E\Big[\chi_{(\alpha(r)=1)}\Big|\alpha(t)=1\Big]\notag\\
& +\lambda_{2,1}(C_3(1)-C_3(2))E\Big[\chi_{(\alpha(r)=2)}\Big|\alpha(t)=1\Big]\Big)\diffns r  \diff s\notag
\end{align}
\begin{align}
=&2C_4(1)u(t)\sigma(t)+2\int_t^Tu(t)\sigma(t) \Big(\lambda_{1,2}(C_4(2)-C_4(1))P(\alpha(r)=1|\alpha(t)=1)\notag\\
& +\lambda_{2,1}(C_4(1)-C_4(2))P(\alpha(r)=2|\alpha(t)=1)\Big)\diffns r+2 C_3(1)u(t)\sigma(t)(T-t)\notag\\
&+2\int_t^T\int_t^su(t)\sigma(t)\Big(\lambda_{1,2}(C_3(2)-C_3(1))P(\alpha(r)=1|\alpha(t)=1)\notag\\
& +\lambda_{2,1}(C_3(1)-C_3(2))P(\alpha(r)=2|\alpha(t)=1)\Big)\diffns r  \diff s.\notag
\end{align}
Using the transition probability for a two-state Markov chain we get
\begin{align}
E\Big[\tilde{q}(t)\Big|\mathcal{F}_t\Big]=&2C_4(1)u(t)\sigma(t)+2u(t)\sigma(t)C_4(2,1)\int_t^T\Big(\lambda_{1,2}\frac{\lambda_{1,2}e^{(\lambda_{1,2} +\lambda_{2,1})(t-r)}+\lambda_{2,1}}{\lambda_{1,2} +\lambda_{2,1}}\notag\\
&-\lambda_{2,1}\frac{\lambda_{1,2}-\lambda_{1,2}e^{(\lambda_{1,2} +\lambda_{2,1})(t-r)}}{\lambda_{1,2} +\lambda_{2,1}}\Big)\diffns r+2 C_3(1)u(t)\sigma(t)(T-t) \notag\\
&+2C_3(2,1)u(t)\sigma(t)\int_t^T\int_t^s\Big(\lambda_{1,2}\frac{\lambda_{1,2}e^{(\lambda_{1,2} +\lambda_{2,1})(t-r)}+\lambda_{2,1}}{\lambda_{1,2} +\lambda_{2,1}}\notag\\
&-\lambda_{2,1}\frac{\lambda_{1,2}-\lambda_{1,2}e^{(\lambda_{1,2} +\lambda_{2,1})(t-r)}}{\lambda_{1,2} +\lambda_{2,2}}\Big)\diffns r \diff s \notag\\
&=2C_4(1)u(t)\sigma(t)+2u(t)\sigma(t)C_4(2,1)\frac{\lambda_{1,2}}{\lambda_{1,2} +\lambda_{2,1}}\Big(1-e^{(\lambda_{1,2} +\lambda_{2,1})(t-T)}\Big)\notag \\
&+2 C_3(1)u(t)\sigma(t)(T-t)+ 2C_3(2,1)u(t)\sigma(t) \frac{\lambda_{1,2}}{\lambda_{1,2} +\lambda_{2,1}}(T-t)\notag\\
&-2C_3(2,1)u(t)\sigma(t)\frac{\lambda_{1,2}}{(\lambda_{1,2} +\lambda_{2,1})^2}\Big(1-e^{(\lambda_{1,2} +\lambda_{2,1})(t-T)}\Big)\notag\\
=&2u(t)\sigma(t)\Big(C_4(1)+C_3(1)(T-t)+C_3(2,1)\frac{\lambda_{1,2}}{\lambda_{1,2} +\lambda_{2,1}}(T-t)\notag\\
&+\frac{\lambda_{1,2}\Big\{C_4(2,1)(\lambda_{1,2} +\lambda_{2,1})-C_3(2,1)\Big\}}{(\lambda_{1,2} +\lambda_{2,1})^2}\Big\{1-e^{(\lambda_{1,2} +\lambda_{2,1})(t-T)}\Big\}\Big).\label{eqqtildecon1}
\end{align}
On the other hand, If $\alpha(t)=e_1$, using the integration by parts formula and the fundamental theorem of calculus, we have
\begin{align}
E\Big[\tilde{r}(t,\zeta)\Big|\mathcal{F}_t\Big]=&2C_4(1)u(t)\gamma(t,\zeta)+2\int_t^Tu(t)\gamma(t,\zeta) \Big(\lambda_{1,2}(C_4(2)-C_4(1))E\Big[\chi_{(\alpha(r)=1)}\Big|\alpha(t)=1\Big]\notag\\
& +\lambda_{2,1}(C_4(1)-C_4(2))E\Big[\chi_{(\alpha(r)=2)}\Big|\alpha(t)=1\Big]\Big)\diffns r+2 C_3(1)u(t)\gamma(t,\zeta)(T-t)\notag\\
&+2\int_t^T\int_t^su(t)\gamma(t,\zeta)\Big(\lambda_{1,2}(C_3(2)-C_3(1))E\Big[\chi_{(\alpha(r)=1)}\Big|\alpha(t)=1\Big]\notag\\
& +\lambda_{2,1}(C_3(1)-C_3(2))E\Big[\chi_{(\alpha(r)=2)}\Big|\alpha(t)=1\Big]\Big)\diffns r  \diff s\notag
\end{align}
\begin{align}
=&2C_4(1)u(t)\gamma(t,\zeta)+2\int_t^Tu(t)\gamma(t,\zeta) \Big(\lambda_{1,2}(C_4(2)-C_4(1))P(\alpha(r)=1|\alpha(t)=1)\notag\\
& +\lambda_{2,1}(C_4(1)-C_4(2))P(\alpha(r)=2|\alpha(t)=1)\Big)\diffns r+2 C_3(1)u(t)\gamma(t,\zeta)(T-t) \notag\\
&+2\int_t^T\int_t^su(t)\gamma(t,\zeta)\Big(\lambda_{1,2}(C_3(2)-C_3(1))P(\alpha(r)=1|\alpha(t)=1)\notag\\
& +\lambda_{2,1}(C_3(1)-C_3(2))P(\alpha(r)=2|\alpha(t)=1)\Big)\diffns r  \diff s.\notag
\end{align}
Similarly, we get
\begin{align}\label{eqrtildecon1}
E\Big[\tilde{r}(t,\zeta)\Big|\mathcal{F}_t\Big]=&2u(t)\gamma(t,\zeta)\Big(C_4(1)+C_3(1)(T-t)+C_3(2,1)\frac{\lambda_{1,2}}{\lambda_{1,2} +\lambda_{2,1}}(T-t)\notag\\
&+\frac{\lambda_{1,2}\Big\{C_4(2,1)(\lambda_{1,2} +\lambda_{2,1})-C_3(2,1)\Big\}}{(\lambda_{1,2} +\lambda_{2,1})^2}\Big\{1-e^{(\lambda_{1,2} +\lambda_{2,1})(t-T)}\Big\}\Big).
\end{align}
Then, the result follows for $\alpha(t)=e_1$. Performing the same computations, one get an expression for $\Gamma(t,T,1)$. This complete the proof.
\end{proof}

The following corollary is a generalization of the result obtained in \cite[Example 4.7]{LZ13}.

\begin{cor} Assume that conditions of \textup{Theorem \ref{thmappli1}} are satisfied. Moreover assume that $C_1,C_2,C_3,C_4:I\rightarrow \mathbb{R}$ satisfy $C_1(1)=-1,C_1(2)=0,C_2(1)=0,C_2(2)=-\frac{1}{2}$, \\$C_3(1)=0,C_3(2)=1,C_4(1)=\frac{1}{2},C_4(2)=1$
Then the optimal control $u^\ast$ for \textup{\eqref{prob111}} satisfies:
\begin{align}\label{equopappli2}
u^\ast(t)=&\dfrac{1}{2\Gamma(t,T,1)\Big(\sigma^2(t)+\int_{\mathbb{R}_0}\gamma^2(t,\zeta) \nu(\diffns \zeta)\Big)}\chi_{\{\alpha(t-)=1\}}+0\times \chi_{\{\alpha(t-)=2\}},
\end{align}
where $\Gamma(t,T,1)=\frac{1}{2}+\frac{\lambda_{1,2}}{\lambda_{1,2} +\lambda_{2,1}}(T-t)+\frac{\lambda_{1,2}\Big\{\frac{1}{2}(\lambda_{1,2} +\lambda_{2,1})-1\Big\}}{(\lambda_{1,2} +\lambda_{2,1})^2}\Big\{1-e^{(\lambda_{1,2} +\lambda_{2,1})(t-T)}\Big\}.$
\end{cor}

\end{appli}

\begin{appli}
We shall now use the results of Section \ref{mallcalappro} to study a problem of recursive utility maximization.
Consider a financial market with two investments possibilities: a risk free asset (bond) with the unit price $S_0(t)$ at time $t$ and a risky asset (stock) with unit price $S(t)$ at time $t$.

Let $r(t)$ be the instantaneous interest rate of the risk free asset at time $t$. If \\$r_t:=r(t,\alpha(t))=\langle \underline{r}|\alpha(t)\rangle$, where $\langle\cdot|\cdot\rangle$ is the usual scalar product in $\mathbb{R}^D$ and \\$\underline{r}=(r_1,r_2, \ldots ,r_D)\in \mathbb{R_+}^D$, then the price dynamic of $S_0$ is given by:
\begin{align}\label{eqbond1}
\diffns S_0(t)=&r(t)S_0(t)\diffns t,\,\,\,S_0(0)=1.
\end{align}
The appreciation rate $\mu(t)$ and the volatility $\sigma(t)$ of the stock at time time $t$ are defined by
\begin{align}\label{dri}
\mu(t):=\mu(t,\alpha(t))=\langle \underline{\mu}|\alpha(t)\rangle, \,\,\,\sigma(t):=\sigma(t,\alpha(t))=\langle \underline{\sigma}|\alpha(t)\rangle\quad t\in [0,T]
\end{align}
 where $\underline{\mu}=(\mu_1,\mu_2, \ldots ,\mu_D)\in \mathbb{R}^D$ and $\underline{\sigma}=(\sigma_1,\sigma_2, \ldots ,\sigma_D)\in \mathbb{R_+}^D$. The stock price process $S$ is described by the following Markov modulated L\'{e}vy process
\begin{align}\label{risk}
 \diffns S(t)=S(t^-)\Big(\mu(t)\diffns t+\sigma(t)\diffns B (t)+\int_{\mathbb{R}\backslash\{0\}}\gamma(t,\zeta)\widetilde{N}_\alpha(\diffns t,\diffns \zeta)\Big),\quad S(0)>0.
\end{align}

Here  $r(t)\geq 0,\,\,\mu(t),\,\,\sigma(t)$ and $\gamma(t,\zeta)>-1+\varepsilon$ \textup{(}for some constant $\varepsilon>0$\textup{)} are given $\mathcal{E}_t$-predictable, integrable processes, where $\left\{\mathcal{E}_t\right\}_{t\in\left[0,T\right]}$ is a given filtration such that \\$\mathcal{E}_t\subset \mathcal{F}_t \text{ for all } t\in [0,T].$

Suppose that, a trader in this market chooses a portfolio $u(t)$, representing the amount she invests in the risky asset at time $t$, then this portfolio is a $\mathcal{E}_t$-predictable stochastic process. Choosing $S_0(t)$ as a numeraire, and setting without loss of generality $r(t)=0$, one can show \textup{(}see \cite{DOPP2011} for such a derivation\textup{)} that the corresponding wealth process $X(t)=X^{(u)}(t)$ satisfies
\begin{equation}\label{portprocess3}
\diffns X(t)=u(t)\Big[\mu(t)\diffns t+\sigma(t)\diffns B(t)+\displaystyle \int_{\mathbb{R}_0}\gamma(t,\zeta)\widetilde{N}_\alpha(\diffns t,\diffns \zeta)\Big],\,\,\,X(0)=x> 0.
\end{equation}
Consider the following stochastic recursive utility, which is given by a Markov switching BSDE.
\begin{align}
 Y (t)  = &  X(T)+ \int_t^T g(s,Y(t),\alpha(s),\omega)\diff s +\int_t^T Z(s)\diff B(s) +\int_t^T  \int_{\mathbb{R}_0}K(s,\zeta)\,\widetilde{N}_\alpha(\diffns \zeta,\diffns s)\notag\\
&+\int_t^T  V(s)\cdot \diffns \widetilde{\Phi}(s), \label{eqassBSDEapp}
\end{align}
where $g:[0,T]\times \mathbb{R} \times \mathbb{S}  \times \mathcal{U} \times \Omega \rightarrow \mathbb{R}$ is such that the BSDE \eqref{eqassBSDEapp} has a unique solution and $(t,\omega) \rightarrow g(t,x,e_i,\omega)$ is $\mathcal{F}_t$-predictable for each given $x$ and $e_i$.  We aim at finding $u^\ast$ and $Y^\ast$ such that $
Y^{(u^\ast)}(0)=\sup_{u\in  \mathcal{A}_{\mathcal{E} }} Y^{(u)}(0)=Y^\ast.$

Assume that $\alpha(t)$ is a two states Markov process and that $g(t,Y(t),\alpha(t),\omega)$ is given by:
\begin{equation}
  g(t,Y(t),1,\omega) =  -c_1(t) Y(t)\ln Y(t) +c_2(t)Y(t) ,\,\,\,g(t,Y(t),2,\omega) = c(t)Y(t)+c_0(t) \label{forwardadjoinappl1}
\end{equation}

Using \textup{Theorem \ref{theomainneccon1G}}, one can show in a similar way as in \cite[Section 5]{OS091},
\begin{thm}
Suppose that $g(t,y,\alpha)$ is as in \textup{\eqref{forwardadjoinappl1}}, $c_1$ is deterministic. Let $\tilde{A} (T)$ be the solution of modified forward adjoint equation and suppose that $\beta$ and $\theta$ satisfy
$$
\mu(t,\alpha) +\sigma(t,\alpha)\beta(t,\alpha)+\int_{\mathbb{R}_0} \gamma(t,\alpha,\zeta) \theta(t,\alpha,\zeta) \nu_\alpha(\diffns \zeta)=0 \text{ for a.a. } t,\omega.
$$
Moreover, assume that $
E\Big[\exp\Big(\int_0^Tc(t) \diffns t\Big)\Big(1+\int_0^T|c_0(t)| \diffns t\Big)\Big]<\infty.$ In addition, suppose that an optimal control $u^\ast$ exists . Then the maximal differential utility is given by:
\begin{align}
  Y^\ast(0,1) = & x\Big(\exp\int_0^Tc_1(t) \diffns t \Big) E[\tilde{A} (T)]\label{slouappliuti1} \\
Y^\ast(0,2) =& xE\Big[\exp\int_0^Tc(t) \diffns t\Big] + \int_0^T E\Big[c_0(t)\exp\int_0^Tc(t) \diffns t\Big]\diffns t
\label{slouappliuti2}
\end{align}

\end{thm}

\begin{proof}
It follows using Theorem \ref{theomainneccon1G} and the arguments in \cite[Section 5]{OS091}.
\end{proof}
\end{appli}


\subsection*{Acknowledgment} The author would like to thank Corina Constantinescu and Apostolos Papaioannou for their helpful comments.


\begin{thebibliography}{99}

\bibitem{Ben83} {\sc Bensousssan, A.} (1983). Maximum principle and dynamic programming Approaches of the optimal control of partially observed diffusions. \emph{Stochastics.} {\bf 9,} 169--222.





\bibitem{Bis78} {\sc Bismut, J. M.} (1978). An introductory approach to duality in optimal stochastic control. \emph{SIAM Review.} {\bf 20,} 62--78.


\bibitem{CoEl10} {\sc Cohen, S.N. and Elliott, R.J.} (2010). Comparisons for backward stochastic differential equations
on Markov chains and related no-arbitrage conditions. \emph{Ann. Appl. Probab.} {\bf 20,} 267-311.

\bibitem{Crep} {\sc Crepey, S.} (2010). \emph{About the pricing equations in finance, in: Paris Princeton Lectures
on Mathematical Finance.} Springer, Berlin.


\bibitem{DOP08} {\sc Di Nunno, G., \O ksendal, B., and Proske, F.} (2008). \emph{Malliavin Calculus
for L\'{e}vy processes with Applications to Finance.} Universitext Springer.

\bibitem{DOPP2011} {\sc Di Nunno, G. \O ksendal, B., Pamen, O.M. and Proske, F.} (2011). A general maximum principle for anticipative stochastic control and applications to insider trading. In \emph{Advanced Mathematical Methods for Finance.} ed. G. Di Nunno and B. \O ksendal. Springer, pp. 181--221

\bibitem{Donel2011} {\sc Donnelly, C.} (2011). Sufficient stochastic maximum principle in a regime-switching
diffusion model. \emph{Applied Mathematic and Optimization.} {\bf 64,} 155–-169.

\bibitem{DonHeu2011} {\sc Donnelly, C. and Heunis, A. J.} (2012). Quadratic risk minimization in a regime-switching model
with portfolio constraints. \emph{SIAM J. Control Optim.} {\bf 50,} 2431--2461.


\bibitem{DE} {\sc Duffie, D. and Epstein, M.} (1992). Stochastic differential utility. \emph{Econometrica.} {\bf 60,} 353--394.



\bibitem{KPQ0} {\sc El Karoui, N., Peng, S. and Quenez, M. C.} (2001). A dynamic maximum principle for the optimization of recursive utilities under constraints. \emph{Ann. Appl. Probab.} {\bf 11,} 664--693.


\bibitem{EAM94} {\sc Elliot, R.J., Aggoun, L. and Moore, J.B.} (1994) \emph{Hidden Markov Models: Estimation and Control.} Springer, New York.





\bibitem{FS06} {\sc Fleming, V. H. and Soner, H. M.} (2006). \emph{Controlled Markov Processes and Viscosity Solutions.} Springer-Verlag.


\bibitem{FOS05} {\sc Framstad, N., \O ksendal, B. and Sulem, A.} (2004). Stochastic maximum
principle for optimal control of jump diffusions and applications to finance. \emph{J. Optimization Theory and Appl.} {\bf 121,} 77--98.

\bibitem{Hamil89} {\sc Hamilton, J.} (1989). A new approach to the economic analysis of non-stationary time series. \emph{Econometrica.} {\bf 57,} 357--384.








\bibitem{Kush72} {\sc Kushner, H. J.} (1972). Necessary conditons for continuous parameter stochastic optimization
problems. \emph{SIAM J. Control Optim.} {\bf 10,} 550--565.


\bibitem{LZ13} {\sc Li, Y. and Zheng, H.} (2013). Weak Necessary and Sufficient Stochastic Maximum Principle
for Markovian Regime-Switching Diffusion Models.  http://arxiv.org/pdf/1210.0371v3.pdf


\bibitem{MMPS13} {\sc Menoukeu-Pamen, O., Meyer-Brandis, T., Proske, F. and Saley, A. B.} (2013). Malliavin calculus applied to optimal control of stochastic partial differential equations with jumps. \emph{Stochastics: An International Journal of Probability and Stochastic Processes} {\bf 85} 631--663.



\bibitem{MOZ12} {\sc Meyer-Brandis, T., \O ksendal, B. and Zhou, X. Y.} (2012). A mean-field stochastic maximum principle via Malliavin calculus. \emph{Stochastics: An International Journal of Probability and Stochastic Processes.} Special Issue: The Mark H.A. Davis festschrift: stochastics, control and finance, {\bf 84} 643--666.

\bibitem{Nua06} {\sc Nualart, D.} (2006). \emph{The Malliavin Calculus and Related Topics.} 2nd edn. Springer.

\bibitem{OS071} {\sc \O ksendal, B. and Sulem, A.} (2007). \emph{Applied Stochastic Control of Jump Diffusions,} 2nd edn. Springer, Berlin.

\bibitem{OS091} {\sc \O ksendal, B. and Sulem, A.} (2009). Maximum principles for optimal control of forward-backward stochastic differential equations with jumps. \emph{SIAM J. Control Optim.} {\bf 48,} 2945--2976.

\bibitem{Pen90} {\sc Peng, S.} (1990). A general stochastic maximum principle for optimal control problems. \emph{SIAM
J. Control and Optim.} {\bf 28} 966--979.

\bibitem{Pen93}{\sc Peng, S.}(1993). Backward Stochastic Differential Equations and Applications to Optimal Control. \emph{ Appl. Math. Optim.} {\bf 27} 125--144.


\bibitem{TW12} {\sc Tao, R. and Wu, Z.} (2012). Maximum principle for optimal control problems of forward-backward regime-switching system and applications. \emph{Systems Control Lett.} {\bf 61,} 911--917.


\bibitem{YZ99} {\sc Yong, J. and Zhou, X. Y.} (1999). \emph{Stochastic controls: Hamiltonian Systems and HJB Equations.} Springer, New York.

\bibitem{ZES2012} {\sc Zhang, X., Elliott, R. J. and Siu, T. K.} (2012). A stochastic maximum principle for a Markov regime-switching jump-diffusion model and its application to finance. \emph{SIAM J.Control Optim.} {\bf 50,} 964--990.

\end{thebibliography}
 \end{document}